\documentclass{amsart}
\usepackage{color}
\usepackage{amsmath, amssymb,enumerate}
\usepackage{cleveref}

\def\l{\left}
\def\r{\right}

\def\mR{\mathbb{R}}
\def\mRd{\mathbb{R}^d}
\def\ve{\varepsilon}
\def\MA{Monge-Amp\`{e}re }
\newcommand{\wt}[1]{\widetilde{#1}}

\DeclareMathOperator{\diam}{diam}
\DeclareMathOperator{\supp}{supp}

\usepackage{soul}

\newtheorem{Theorem}{Theorem}[section]
\newtheorem{Lemma}[Theorem]{Lemma}
\newtheorem{Proposition}[Theorem]{Proposition}
\newtheorem{Corollary}[Theorem]{Corollary}
\newtheorem{remark}[Theorem]{Remark}

\newtheorem{Assumption}[Theorem]{Assumption}

\usepackage{algpseudocode,algorithm,algorithmicx}

\algrenewcommand\algorithmicrequire{\textbf{Input:}}
\algrenewcommand\algorithmicensure{\textbf{Output:}}

\numberwithin{equation}{section}

\begin{document}

\title[Error Estimates for Optimal Transport]{Quantitative Stability and Error Estimates for
  Optimal Transport Plans}

\author{Wenbo Li}
\address[Wenbo Li]{Department of Mathematics, The University of Tennessee, Knoxville, Tennessee 37996}
\email[Wenbo Li]{wli50@utk.edu}

\author{Ricardo H. Nochetto}
\address[Ricardo H. Nochetto]{Department of Mathematics and Institute for Physical Science and Technology, University of Maryland, College Park, Maryland 20742}
\email[Ricardo H. Nochetto]{rhn@umd.edu}
\thanks{R.H. Nochetto was partially supported by the NSF Grants DMS-1411808 and
  DMS-1908267. W. Li was partially supported by NSF Grant DMS-1411808 and the Patrick and Marguerite Sung Fellowship in Mathematics of the University of Maryland.}

\date{Draft version of \today}

\maketitle

\begin{abstract}
Optimal transport maps and plans between two absolutely continuous measures $\mu$ and $\nu$ can be approximated by solving semi-discrete or fully-discrete optimal transport problems. These two problems ensue from approximating $\mu$ or both $\mu$ and $\nu$ by Dirac measures. Extending an idea from \cite{Gigli-Holder:11}, we characterize how transport plans change under perturbation of both $\mu$ and $\nu$. We apply this insight to prove error estimates for semi-discrete and fully-discrete algorithms in terms of errors solely arising from approximating measures. We obtain weighted $L^2$ error estimates for both types of algorithms with a convergence rate $O(h^{1/2})$. This coincides with the rate in \cite[Theorem 5.4]{Berman-Convergence:18} for semi-discrete methods, but the error notion is different.
\end{abstract}


\vspace{0.2cm}



\section{Introduction}
The optimal transportation problem (OT), first proposed by Monge \cite{Monge-memoire:1781} and later generalized by Kantorovich \cite{Kantorovich-translocation:1942}, has been extensively studied from the theoretical point of view \cite{Brenier-Decomposition:87,Brenier-Least;89,caffarelli1992boundary,caffarelli1992regularity,caffarelli1996boundary,Gigli-Holder:11,Santambrogio-Optimal:15,Villani-Optimal:08}.
It has a wide variety of applications in economics \cite{Chiappori-Hedonic;10, Beiglbock-Model;13}, fluid mechanics \cite{Brenier-Least;89}, meteorology \cite{Cullen-Applications;99, Cullen-Fully;03}, image processing \cite{Rubner-Earth;00, Rabin-Regularization;10}, transportation \cite{Xia-Optimal;03, Carlier-Variational;05}, and optics design \cite{Gutierrez-Refractor;09, Prins-Monge;14}. In this section, we briefly introduce the basic theory and numerical methods of OT, and point out our contribution to the numerical analysis of OT in this paper.

\subsection{Introduction to Optimal Transport}\label{SS:OT-IntroOT}
If $X, Y$ are subdomains of $\mRd$ and $\mu \in P(X), \nu \in P(Y)$ are given probability measures, the Monge formulation of OT is to find an optimal map $T: X \to Y$
which minimizes the transport cost, i.e.
\begin{equation} \label{E:Monge}
\inf_{T: T_{\#}\mu = \nu} \; \int_{X} c\l(x,T(x)\r) \; d\mu(x),
\end{equation}
where the cost function $c(x,y): X \times Y \to [0,\infty)$ is given. Hereafter, $T_{\#}\mu$ denotes the push-forward of measure $\mu$ through $T$, namely $T_{\#}\mu = \nu$ means that for any measurable set $A \subset Y$, we have $\nu(A) = \mu(T^{-1}(A))$, or equivalently
\begin{equation}\label{eq:change-variables}
\int_X \phi(T(x)) \, d\mu(x) = \int_Y \phi(y) \, d\nu(y)
\end{equation}
for all continuous functions $\phi:Y\to\mathbb{R}$.
Since this problem is difficult to study, and sometimes the optimal map does not exist, Kantorovich \cite{Kantorovich-translocation:1942} generalized the notion of transport map, and considered the following problem:
\begin{equation} \label{E:Kantorovich}
\inf_{\gamma \in \Pi(\mu,\nu)} \; \int_{X \times Y} c(x,y) \; d\gamma(x,y),
\end{equation}
where $\Pi(\mu,\nu)$ is the set of transport plans between $\mu$ and $\nu$, namely
\[
\Pi(\mu,\nu) := \Big \{  \gamma \in P(X \times Y) \; : \;
(\pi_1)_{\#}\gamma = \mu,\; (\pi_2)_{\#}\gamma = \nu \Big \}.
\]
Here $\pi_1$ and $\pi_2$ are the projections defined as $\pi_1(x,y) = x, \pi_2(x,y) = y$. This definition implies that for $\gamma \in \Pi(\mu,\nu)$ and any measurable sets $A \subset X, B \subset Y$, we have $\gamma(A \times Y) = \mu(A), \gamma(X \times B) = \nu(B)$. For $X=Y=\mRd$ and $1 \le p < \infty$, let $P_p(\mRd) \subset P(\mRd)$ be the set of probability measures with bounded $p$ moment, i.e.
\begin{equation}\label{E:second-moment}
P_p(\mRd) := \left\{ \mu \in P_p(\mRd) \colon \int_{\mRd} |x|^p d\mu(x) < \infty \right\},
\end{equation}
which clearly contains those probabilities measures with bounded support. If $\mu, \nu \in P_p(\mRd)$, then $\gamma\in P_p(\mR^{2d})$ for all $\gamma\in\Pi(\mu,\nu)$ because
\begin{equation}\label{p-moment} 
\int_{\mR^{2d}} \big( |x|^p + |y|^p \big) \, d\gamma(x,y) =
\int_{\mR^d} |x|^p \, d\mu(x) + \int_{\mR^d} |y|^p \, d\nu(y) < \infty.
\end{equation}
Moreover, if $c(x,y) = |x-y|^p$, the Kantorovich problem \eqref{E:Kantorovich} always has a finite minimum value. In fact, this defines the well-known Wasserstein metric on $P_p(\mRd)$:
\[
W_p(\mu, \nu) := \l( \min_{\gamma \in \Pi(\mu,\nu)} \; \int_{\mRd \times \mRd} |x-y|^p \; d\gamma(x,y) \r)^{1/p} \quad \forall \, \mu, \nu \in P_p(\mRd).
\]
In addition, for quadratic cost, i.e. $p = 2$, \cite{Brenier-Decomposition:87} shows that there exists a unique optimal map $T = \nabla \varphi$ for a convex function $\varphi$ ($T$ is uniquely determined $\mu$-a.e.) provided that $\mu$ 
gives no mass to $(d-1)$-surfaces of class $C^2$. If $\mu, \nu$ are absolutely continuous measures with respect to Lebesgue measure with densities $f,g\ge0$, then
\[
\mu(S) = \int_S f(x) dx = \nu(\nabla\varphi(S)) = \int_{\nabla\varphi(S)} g(y) dy =
\int_S g(\nabla \varphi(x)) \det D^2 \varphi(x) dx,
\]
whence $\varphi$ satisfies the generalized Monge-Amp\`ere equation
\begin{equation}\label{MA}
g(\nabla\varphi(x)) \det D^2 \varphi(x) = f(x)
\quad \forall \,x\in X,
\end{equation}
with second type boundary condition $\nabla\varphi(X) = Y$ \cite[section 4.6]{Figalli-MA:2017} provided that the domains $X,Y \subset \mRd$ are chosen so that:
\begin{equation}\label{eq:def-XY}
X = \{ x \in \mRd: f(x)>0 \}, \quad Y = \{ y \in \mRd : g(y) > 0 \}.
\end{equation}
The optimal transport map $T$ induces an optimal transport plan $\gamma\in P(X,Y)$ given by $\gamma = (id,T)_{\#}\mu $, where $id$ denotes the identity map, namely
\[
\gamma(A) = \mu \big\{x\in X: \quad (x,T(x)) \in A  \big\}
\]
for all measurable sets $A \subset X\times Y$.
Similar results also hold for any $p \in (1,\infty)$, and for general costs $c(x,y)$ satisfying a twist condition \cite[Chapter 10]{Villani-Optimal:08}.

\subsection{Numerical Methods for OT and Our Contribution}\label{S:Numerical-OT}
In view of the numerous and diverse applications of OT, developing fast and sound numerical methods for OT is of paramount importance. Several algorithms do exist, ranging from those inspired by PDEs and variational techniques for absolutely continuous measures $\mu$ and $\nu$ \cite{Angenent-Minimizing:03, Benamou-Computational:00, Benamou-Minimal:17, Benamou-Numerical:14, Froese-Numerical:12, Hamfeldt-Viscosity:19, Lindsey-Optimal:17, Papadakis-Optimal:14} to those approximating one or both measures by Dirac masses and then solving the approximate OT \cite{Benamou-Iterative:15, Burkard-Assignment:12, Cuturi-Sinkhorn:13, Kitagawa-Convergence:16, Levy-Numerical:15, Merigot-Multiscale:11, Oberman-Efficient:15, Schmitzer-Sparse:16, Schmitzer-Hierarchical:13}. However, intrinsic difficulties make their numerical analysis far from complete.

In this paper, we develop stability and error analyses for quadratic costs that accounts for the effect of approximating measures $\mu$ and $\nu$ by Dirac masses. To this end, we extend the stability estimates of Gigli \cite{Gigli-Holder:11}, originally derived for optimal maps, to optimal plans $\gamma$. We do not develop new techniques to approximate $\gamma$.

If at least one of the two measures is discrete, then it is possible to solve for the exact transport map numerically without further approximations, which
includes semi-discrete algorithms \cite{Aurenhammer-Minkowski:98, Kitagawa-Convergence:16, Levy-Numerical:15,  Merigot-Multiscale:11} and fully-discrete methods \cite{Burkard-Assignment:12,Schmitzer-Sparse:16, Schmitzer-Hierarchical:13}. For these methods, since their errors solely come from approximating absolutely continuous measures with discrete measures, our results directly lead to error estimates. We also point out that an error estimate for a semi-discrete method was recently obtained by Berman \cite{Berman-Convergence:18}, but there are no such results for fully-discrete schemes.

We now describe our results. Let $X, Y$ be compact sets of $\mRd$ and $\mu, \nu$ be absolutely continuous measures with respect to the Lebesgue measure with densities $f\in L^1(X), g\in L^1(Y)$, respectively. Let
\[
\mu_h = \sum_{i=1}^N f_i \delta_{x_i}, \quad \nu_h = \sum_{j=1}^M g_j \delta_{y_j}
\]
be approximations of $\mu, \nu$ governed by a parameter $h>0$, which means the Wasserstein distances satisfy
\[
W_2(\mu,\mu_h)\le h,\quad W_2(\nu, \nu_h)\le h.
\]
To make this more concrete, we briefly introduce one way to obtain an approximation $\mu_h$ of $\mu$.
Choose $(x_i)_{i=1}^N$ such that $X \subset \cup_{i=1}^N B_h(x_i)$, where $B_r(x)$ denotes the open ball with radius $r$ centered at $x$. Then consider Voronoi tessellations: let
\begin{equation}\label{eq:Voronoi}
V_i := \Big\{x \in X: |x-x_i| = \min_{1\le j \le N} |x-x_j| \Big\}
\end{equation}
be the Voronoi cell for $x_i$ and
\[
f_i := \mu(V_i), \qquad \mu_h := \sum_{i=1}^N f_i \delta_{x_i}.
\]
Notice that we may assume $f_i > 0$ since otherwise for $f_i = 0$ we could just drop the Dirac measure at $x_i$. Define the map $U_h: X \to (x_i)_{i=1}^N$ such that $U_h(x) = x_i$ for all $x \in V_i$. It can be easily checked that this map $U_h$ is well defined and satisfies $|U_h(x) - x| \le h$ a.e. in $X$ because the intersection between $V_i$ and $V_j$ for $i \neq j$ is of zero Lebesgue measure and $X \subset \cup_{i=1}^N B_h(x_i)$; in particular $V_i \subset B_h(x_i)$ for all $1 \le i \le N$. Therefore $U_h$ is a transport map from $\mu$ to $\mu_h$ and thus
\begin{equation}\label{map-mu-muh}
W_2(\mu, \mu_h) \le \l( \int_{X} |x-U_h(x)|^2 \; d\mu(x) \r)^{1/2} \le \l( \int_{X} h^2 \; d\mu(x) \r)^{1/2} = h.
\end{equation}
Similarly, we could approximate the absolutely continuous measure $\nu$ with density $g$ and bounded support $Y$ with $\nu_h = \sum_{j=1}^M g_j \delta_{y_j}$ satisfying $W_2(\nu, \nu_h) \le h$.

The semi-discrete algorithms solve the OT between $\mu_h$ and $\nu$ upon finding a nodal function $\varphi_h: (x_i)_{i=1}^N \to \mR$ such that
\begin{equation}
f_i = \nu(F_i) = \int_{F_i} g(y)dy, \qquad F_i := \partial \varphi_h(x_i) \cap Y,
\end{equation}
where the discrete subdifferential is given by
\[
\partial \varphi_h(x_i) := \big\{y \in \mRd: \quad \varphi_h(x_j) \ge \varphi_h(x_i) + y \cdot (x_j - x_i) \quad \forall j =1, \cdots, N\big \}.
\]
This type of discretization was introduced by Oliker and Prussner \cite{OlikerPrussner-MA:1989} for Dirichlet boundary conditions whereas error estimates have been derived in \cite{NochZhang-MA:2016,NeilanZhang-MAW2p:2018}. The set $F_i$ coincides with the subdifferential of the convex envelope $\Gamma(\varphi_h)$ associated with $\varphi_h$ \cite[Lemma 2.1]{NochZhang-MA:2016}. The function $\Gamma(\varphi_h)$ is piecewise linear and induces a mesh with nodes $(x_i)_{i=1}^N$ whose elements may be quite elongated; see \cite[Section 2.2]{NochZhang-MA:2016}. We also refer to Berman \cite{Berman-Convergence:18}, who has derived error estimates for the second type boundary condition involving $\Gamma(\varphi_h)$.

Denote the barycenter of $\partial\varphi_h(x_i)$ with respect to the measure $\nu$ by $m_i$, namely $m_i:= f_i^{-1}\int_{F_i} yg(y) \, dy$, and define the map $T_h$ such that $T_h(x_i) = m_i$ for all $1 \le i \le N$. Under proper assumptions on measures $(\mu, \nu)$, to be stated in Section \ref{S:assumptions}, we prove a weighted $L^2$ error estimate in \Cref{C:OT-EE-Semi} for the exact optimal transport map $T$
\begin{equation}\label{eq:semi-discrete}
\l( \sum_{i=1}^N f_i \; |T(x_i) - T_h(x_i)|^2 \r)^{1/2} \le C(\mu, \nu) h^{1/2}.
\end{equation}
This rate of convergence coincides with that in \cite[Theorem 5.4]{Berman-Convergence:18} for $\nabla\Gamma(\varphi_h)$, but the error notion is different; we refer to Section \ref{S:EE-semi} for details. Our approach is taylored to discrete transport plans and thus extends to fully-discrete methods.

Fully-discrete methods aim to find the discrete optimal transport plan
\[
\gamma_h = \sum_{i=1}^N \sum_{j=1}^M \gamma_{h,ij} \, \delta_{x_i} \, \delta_{y_j}
\]
between $\mu_h$ and $\nu_h$ through the constrained minimization problem
\begin{equation}\label{eq:pb-fully-discrete}
\min_{\gamma_h} \sum_{i,j=1}^{N,M} \gamma_{h,ij} c_{ij} : \quad
\gamma_{h,ij} \ge 0, \;\; \sum_{i=1}^N \gamma_{h,ij} = g_{j}, \;\; \sum_{j=1}^M \gamma_{h,ij} = f_{i}.
\end{equation}
If we construct the map $T_h(x_i) := \frac{1}{f_i} \sum_{j=1}^M \gamma_{h,ij} y_j$ from $\gamma_h$ for $1\le i \le N$, then we also obtain a weighted $L^2$ error estimate in \Cref{T:OT-EE-Fully} for the optimal map $T$
\begin{equation}\label{eq:fully-discrete}
\l( \sum_{i=1}^N f_i \, |T(x_i) - T_h(x_i)|^2 \r)^{1/2} \le C(\mu, \nu) \, h^{1/2},
\end{equation}
under suitable assumptions on measures $(\mu, \nu)$ described in Section \ref{S:assumptions}. This is a new error estimate for fully-discrete schemes, and the convergence rate in \eqref{eq:fully-discrete} is the same as \eqref{eq:semi-discrete} for semi-discrete methods.

\subsection{Outline}
In Section \ref{S:assumptions} we introduce the notion of $\lambda$-regularity and show that it leads to $W^1_\infty$-regularity of the transport map $T$. We prove in Section \ref{S:Q-Stability} that $\lambda-$regularity implies a stability bound characterizing how the optimal transport plan $\gamma$ changes under perturbations of both $\mu$ and $\nu$; this hinges on an idea from \cite{Gigli-Holder:11}. We measure the change of transport plans using either a weighted $L^2$ norm in \Cref{th:OT_stability} or the Wasserstein metric in \Cref{cor:Wass_prod_space}. We apply the stability bound to derive error estimates in Sections \ref{S:EE-semi} and \ref{S:EE-fully}. For semi-discrete schemes, we obtain weighted $L^2$ error estimates in \Cref{th:EE-Semi} and \Cref{C:OT-EE-Semi} of Section \ref{S:EE-semi} with a convergence rate $O(h^{1/2})$. We also compare our geometric estimate of \Cref{C:OT-EE-Semi} with a similar one due to Berman \cite{Berman-Convergence:18}; however, the two error notions are different. Moreover, we obtain in \Cref{T:OT-EE-Fully} of Section \ref{S:EE-fully} an entirely new error estimate of order $O(h^{1/2})$ for fully-discrete schemes with errors measured in both the weighted $L^2$ norm and the Wasserstein distance. 

\section{Regularity and Nondegeneracy} \label{S:assumptions}

We now discuss the assumptions we make on measures $\mu, \nu \in P_2(\mRd)$ in order to prove quantitative stability bounds for optimal transport plans and error estimates for numerical methods.

\begin{Assumption}[$\lambda$-regularity]\label{l-regularity}
We say that $(\mu,\nu,\varphi)$ is $\lambda-$\textit{regular} for $\lambda > 0$ if
$\varphi: \mRd \to \mR$ is a convex function such that $\nabla \varphi$ is the optimal transport map from $\mu$ to $\nu$ and $\varphi$ is $\lambda-$smooth in the sense that:
\begin{equation}\label{eq:l-smooth}
\varphi((1-t)x_0 + tx_1) + \dfrac{t(1-t)\lambda}{2} |x_0 - x_1|^2 \ge (1-t)\varphi(x_0) + t\varphi(x_1),
\end{equation}
for any $x_0,x_1 \in \mRd$ and $t \in [0,1]$, i.e. $\frac{\lambda}{2}|x|^2 - \varphi(x)$ is convex. We also say $(\mu,\nu)$ is $\lambda-$regular if there exists a $\varphi$ such that $(\mu,\nu,\varphi)$ is $\lambda-$regular.
\end{Assumption}

We will see below that \eqref{eq:l-smooth} implies $\varphi\in W^2_\infty(\mathbb{R}^d)$; note that if $\varphi\in C^2(\mathbb{R}^d)$, then $D^2 \varphi(x) \le \lambda I$ for all $x\in \mathbb{R}^d$. Although we require $\varphi$ to be defined in the whole $\mRd$ in \Cref{l-regularity}, this is not restrictive because any convex function $\varphi:X\to\mathbb{R}$ defined in a bounded convex set $X$ can be extended as in \cite[Theorem 4.23]{Figalli-MA:2017}
$$
E{\varphi}(z) = \sup_{x \in X, p \in \partial\varphi(x)} \varphi(x) + \langle p, z-x \rangle.
$$
Then one can show that $E{\varphi} = \varphi$ in $X$, and $E{\varphi}$ satisfies \eqref{eq:l-smooth} for any $x_0,x_1 \in \mRd$ if \eqref{eq:l-smooth} holds for $\varphi$ and any $x_0, x_1 \in X$.

Now we introduce the {\it Legendre transform} $\varphi^*$ of $\varphi$, which is defined by
\begin{equation}\label{eq:legendre-def}
\varphi^*(y) = \sup_{z \in \mRd} \langle y,z \rangle - \varphi(z).
\end{equation}
Since $\varphi$ is convex, given $y \in \partial\varphi(x)$ we readily get for all $z\in\mathbb{R}^d$
\[
\varphi(x) + \langle y, z-x \rangle \le \varphi(z)
\quad\Rightarrow\quad
\langle y,z \rangle - \varphi(z) \le \langle y,x \rangle - \varphi(x)
\]
whence in view of \eqref{eq:legendre-def} the following two key properties of $\varphi^*$ are valid:
$$
\varphi^*(y) = \langle y,x \rangle - \varphi(x) < +\infty
\quad\forall \, y\in\partial\varphi(x),
$$
and $y \in \partial \varphi(x)$ if and only if $x \in \partial\varphi^*(y)$. To see that $y \in \partial \varphi(x)$ implies $x \in \partial\varphi^*(y)$, let $z \in\mRd$ be arbitrary and notice that \eqref{eq:legendre-def} for $\varphi^*(z)$ yields
\[
\varphi^*(y) + \langle x, z-y \rangle = \langle x,z \rangle - \varphi(x) \le \varphi^*(z)
\quad\Rightarrow\quad
x \in \partial\varphi^*(y).
\]
Consequently, if $\varphi$ and $\varphi^*$ are of class $C^1$, then $\nabla\varphi^*=(\nabla\varphi)^{-1}$ is the inverse of the transport map $\nabla\varphi$. Moreover,
\begin{equation}\label{eq:finiteness-phi*}
-\infty < \int_Y \varphi^*(y) \, d\nu(y) < +\infty,
\end{equation}
provided $(\mu,\nu,\varphi)$ is $\lambda$--regular. In fact, since $(\nabla \varphi)_{\#}\mu = \nu$, in view of \eqref{eq:change-variables} we have
$$
\int_Y \varphi^*(y) d\nu(y) = \int_X \varphi^*(\nabla \varphi(x) ) d\mu(x) = 
\int_X \Big( \langle \nabla \varphi(x),x \rangle - \varphi(x) \Big)\; d\mu(x).
$$
We exploit the convexity of $\varphi$ and $\frac{\lambda}{2} |\cdot|^2 - \varphi$ to obtain
$$
\varphi(0) - \frac{\lambda}{2} |x|^2 \le \varphi(x) - \langle \nabla \varphi(x), x \rangle \le \varphi(0).
$$
Since $\mu \in P_2(\mRd)$, \eqref{E:second-moment} imply that the following integrals are finite and yield \eqref{eq:finiteness-phi*}
$$
- \int_X \varphi(0) \, d\mu(x)
\le \int_X \Big( \langle \nabla \varphi(x),x \rangle - \varphi(x) \Big) \, d\mu(x)
\le \int_X \Big( -\varphi(0) + \dfrac{\lambda}{2} |x|^2 \Big) \, d\mu(x).
$$

The following lemma states that $\lambda$--regularity of a convex function $\varphi$ is equivalent to uniform convexity of its Legendre transform $\varphi^*$ \cite[Proposition 2.6]{AzePenotConvex}.

\begin{Lemma}[$\lambda$--regularity vs uniform convexity]\label{L:l-regularity}
If $\varphi: \mRd \to \mR$ is convex, then the following statements are valid:
\begin{enumerate}[(a)]
\item If $\varphi$ is $\lambda-$smooth, then its Legendre transform $\varphi^*$ must satisfy
\begin{equation}\label{eq:l-uniform-convex}
\varphi^*((1-t)y_0 + ty_1) + \dfrac{t(1-t)}{2\lambda} |y_0 - y_1|^2 \le (1-t)\varphi^*(y_0) + t\varphi^*(y_1)
\end{equation}
for all $y_0,y_1\in\mathbb{R}^d$ and $t \in [0,1]$, i.e. $\varphi^*(y) - \frac{1}{2\lambda}|y|^2$ is convex.

\item If $\varphi(y) - \frac{1}{2\lambda}|y|^2$ is convex, then its Legendre transform $\varphi^*$ is $\lambda-$smooth.
\end{enumerate}  
\end{Lemma}
\begin{proof} For completeness, we repeat the proof of \cite[Proposition 2.6]{AzePenotConvex} to show (a); the proof of (b) is similar. Given $y_0, y_1 \in \mRd$ and $t \in [0,1]$, for any $x_0, v \in \mRd$ we set $x_t := x_0 + tv$ and $y_t := y_0 + t(y_1 - y_0)$ and use \eqref{eq:legendre-def} to write
	\begin{align*}
	(1-t)\varphi^*(y_0) &+ t\varphi^*(y_1) \ge (1-t)\langle y_0,x_0 \rangle + t\langle y_1,x_1 \rangle - (1-t)\varphi(x_0) - t\varphi(x_1)\\
	&\ge (1-t)\langle y_0,x_0 \rangle + t\langle y_1,x_1 \rangle -\varphi(x_t) - \dfrac{t(1-t)\lambda}{2} |v|^2 \\
	&= \langle y_t,x_0+tv \rangle - \varphi(x_0+tv) + t(1-t)\langle y_1 - y_0,v \rangle - \dfrac{t(1-t)\lambda}{2} |v|^2.
	\end{align*}
We exploit that $x_0, v$ are arbitrary. Taking supremum with respect to $x_0$, we get
	$$
	(1-t)\varphi^*(y_0) + t\varphi^*(y_1) \ge \varphi^*(y_t) + t(1-t)\langle y_1 - y_0,v \rangle - \dfrac{t(1-t)\lambda}{2} |v|^2.
	$$
	Maximizing the last two terms with respect to $v$, we obtain $v=\frac{1}{\lambda}(y_1-y_0)$ and 
	$$
	(1-t)\varphi^*(y_0) + t\varphi^*(y_1) \ge \varphi^*(y_t) + \dfrac{t(1-t)}{2\lambda} |y_0 - y_1|^2,
	$$
	which is the asserted inequality \eqref{eq:l-uniform-convex}. That $\varphi^*(y)-\frac{1}{2\lambda} |y|^2$ is convex is a straightforward calculation that completes the proof.
\end{proof}

\begin{Lemma}[$W^2_\infty$--regularity]\label{L:regularity}
A $\lambda$--smooth convex map $\varphi$ is of class $W^2_\infty(\mathbb{R}^d)$ and the Lipschitz constant of $T=\nabla\varphi$ is $\lambda$, namely
\begin{equation}\label{E:Lip-T} 
|T(x_1) - T(x_2)| \le \lambda |x_1 - x_2| \quad\forall \, x_1,x_2\in\mathbb{R}^d.
\end{equation}
\end{Lemma}
\begin{proof}
  According to \Cref{L:l-regularity}(a), the function $\psi(y) := \varphi^*(y) - \frac{1}{2\lambda} |y|^2$ is convex. If $x_1\in\partial\varphi^*(y_1)$ and $x_2\in\partial\varphi^*(y_2)$, then
\begin{gather*}
  x_1-\frac{1}{\lambda} y_1 \in \partial\psi(y_1)
  \quad\Rightarrow\quad
  \psi(y_2) \ge \psi(y_1) + \langle x_1-\frac{1}{\lambda} y_1, y_2-y_1 \rangle,
  \\
  x_2-\frac{1}{\lambda} y_2 \in \partial\psi(y_2)
  \quad\Rightarrow\quad
  \psi(y_1) \ge \psi(y_2) + \langle x_2-\frac{1}{\lambda} y_2, y_1-y_2 \rangle.
\end{gather*}
Adding the two inequalities and rearranging terms yields
\[
\frac{1}{\lambda} \, |y_1-y_2|^2 \le \langle y_1-y_2, x_1 - x_2\rangle \le
|y_1-y_2| \, |x_1 - x_2|.
\]
This implies the assertion \eqref{E:Lip-T}.
\end{proof}

It is now apparent from \Cref{L:regularity} that the constant $\lambda$ dictates the stability of the transport map $T=\nabla\varphi$; $\lambda$ is also the stability constant that appears in our error estimates of Sections \ref{S:EE-semi} and \ref{S:EE-fully}.
If $\varphi \in C^2(\mRd)$, then \eqref{E:Lip-T} is equivalent to $D^2 \varphi(x) \le \lambda I$ for all $x$. As we have already mentioned at the end of \Cref{SS:OT-IntroOT}, the optimal map from $\mu$ to $\nu$ for quadratic cost is given by $\nabla \varphi$ under suitable assumptions on $\mu$. 
By Caffarelli's regularity results \cite{caffarelli1992boundary, caffarelli1992regularity, caffarelli1996boundary}, if both $\overline{X}=\supp(\mu)$ and $\overline{Y}=\supp(\nu)$ are uniformly convex domains of $\mRd$ with $C^2$ boundary and the densities $f,g$ of $\mu, \nu$ are bounded away from $0$ and $f \in C^{0,\alpha}(\overline{X}), g\in C^{0,\alpha}(\overline{Y})$, then the solution $\varphi$ of the corresponding second boundary value \MA problem \eqref{MA} is of class $C^{2,\alpha}(\overline{X})$. This implies that
$(\mu,\nu,\varphi)$ is $\lambda-$regular for some $\lambda > 0$ if we extend $\varphi$ to $\mRd$. Moreover, the same assumptions imply that $(\nu,\mu,\varphi^*)$ is $\xi-$\textit{regular} for the Legendre transform $\varphi^* \in C^{2,\alpha}(\overline{Y})$  of $\varphi$ and some $\xi > 0$.

\section{Quantitative Stability of Optimal Transport Plans} \label{S:Q-Stability}
In this section, we generalize the quantitative stability bounds of Gigli \cite[Corollary 3.4]{Gigli-Holder:11} for optimal transport maps, and show some consequences of our theorem under suitable assumptions on measures $\mu$ and $\nu$.  They will be useful in Sections \ref{S:EE-semi} and \ref{S:EE-fully} to derive error estimates. The stability estimate in \cite{Gigli-Holder:11} is based on bounding the $L^2$ difference between an optimal transport map and another feasible transport map by the difference between their transport costs. Proposition \ref{prop:OT_stability_fix_measures} below is just a generalization of this important property in that it replaces transport maps by transport plans.

\begin{Proposition}[stability of transport plans]
	\label{prop:OT_stability_fix_measures}
	Let $(\mu,\nu,\varphi)$ be $\lambda-$regular for $\lambda > 0$, and $T = \nabla \varphi$ be the optimal transport map from $\mu$ to $\nu$. Then, for any transport plan $\gamma \in \Pi(\mu,\nu)$, there holds
        \begin{equation}
	\begin{aligned}\label{eq:OT_stability_fix_measures}
	  \int_{X \times Y} |y &- T(x)|^2 d\gamma(x,y)
          \\ & \le \lambda \l( \int_{X \times Y} |y-x|^2 d\gamma(x,y) - \int_{X} |T(x) - x|^2 d\mu(x) \r).
	\end{aligned}
        \end{equation}
\end{Proposition}

\begin{proof}
  We proceed in the same way as in \cite{Gigli-Holder:11}. If $\varphi^*$ is the Legendre transform of $\varphi$, then $\int_Y \varphi^*(y) d\nu$ is finite from \eqref{eq:finiteness-phi*}. Since $T_{\#} \mu = \nu$, using \eqref{eq:change-variables}, we have
\begin{align*} 
  0 = \int_Y \varphi^*(y) d\nu(y) - \int_X \varphi^*(T(x)) d\mu(x) =
  \int_{X \times Y} \l( \varphi^*(y) - \varphi^*(T(x)) \r) \ d\gamma(x,y).
\end{align*}
In view of \Cref{L:regularity} ($W^2_\infty$-regularity), the map $T=\nabla\varphi$ is
of class $W^1_\infty$. It thus follows that $x = \nabla\varphi^* (T(x))$ for all $x\in X$ because $\nabla\varphi^* = (\nabla\varphi)^{-1}$. Using 
	\[
	2 \langle z, z-y \rangle = |z|^2 - |y^2| + |z-y|^2 \quad \forall y,z \in \mRd,
	\]
together with \Cref{L:l-regularity} ($\lambda$-regularity vs uniform convexity), namely $x \mapsto \varphi^*(x) - \frac{1}{2\lambda}|x|^2$ is convex, we further obtain that
        \begin{align*}
	0 &=  \int_{X \times Y} \l( \varphi^*(y) - \varphi^*(T(x)) \r) \ d\gamma(x,y) \\
	&\ge \int_{X \times Y} \langle x, \ y - T(x) \rangle \ d\gamma(x,y) + \frac{1}{2\lambda} \int_{X \times Y} |y - T(x)|^2 \ d\gamma(x,y).
	\end{align*}        
	Noticing that
        \[
        2 \langle x, \ y - T(x) \rangle = |T(x) - x|^2 - |y - x|^2 - |T(x)|^2 + |y|^2
        \]
we deduce
	\begin{align*}
	2\int_{X \times Y} \langle x , \ y &- T(x) \rangle \ d\gamma(x,y)
	=  \int_X |T(x) - x|^2 d\mu(x)
        \\ &-  \int_{X \times Y} |y - x|^2 d\gamma(x,y) - \int_X |T(x)|^2 d\mu(x) + \int_Y |y|^2 d\nu(y).
        \end{align*}
        In view of \eqref{eq:change-variables} the last two terms are equal and cancel out. Consequently,
        \begin{equation*}
        2\int_{X \times Y} \langle x , \ y - T(x) \rangle \ d\gamma(x,y) = 
        \int_X |T(x) - x|^2 d\mu(x) - \int_{X \times Y} |y - x|^2 d\gamma(x,y) ,
	\end{equation*}
whence
        \[
	0 \ge \int_{X} |T(x) - x|^2 d\mu - \int_{X\times Y} |y - x|^2 d\gamma(x,y) + \frac{1}{\lambda} \int_{X\times Y} |y - T(x)|^2 \ d\gamma(x,y).
	\]
Rearranging the equation above gives \eqref{eq:OT_stability_fix_measures}.
\end{proof}

\begin{remark}[interpretation of \eqref{eq:OT_stability_fix_measures}]
  Notice that the right hand side in \eqref{eq:OT_stability_fix_measures} without factor $\lambda$ is the difference of transport costs between the given transport plan $\gamma$ and the optimal transport map $T$. The factor $\lambda$ acts as a stability constant.
\end{remark}

\begin{remark}[stability of transport maps]\label{rem:l2-error-plan}
	We point out that the left hand side of \eqref{eq:OT_stability_fix_measures} can be seen as the square of a weighted $L^2$-error between the transport plan $\gamma$ and the optimal map $T$. To understand this, notice that if the plan $\gamma$ in \Cref{prop:OT_stability_fix_measures} is induced by a map $S: \mRd \to \mRd$, i.e. $(id, S)_{\#} \mu = \gamma$, then \eqref{eq:OT_stability_fix_measures} can be written as
	\[
	\int_X |S(x) - T(x)|^2 d\mu(x) \le \lambda \l( \int_X |S(x)-x|^2 d\mu(x) - \int_X |T(x) - x|^2 d\mu(x) \r) ,
	\]
	which is the same as \cite[Proposition 3.3]{Gigli-Holder:11}. This is an estimate of $\Vert S - T \Vert^2_{L^2_{\mu}(X)}$.
\end{remark}

Let us recall the gluing lemma for measures (see \cite[p.23]{Villani-Optimal:08}), which plays an important role in proving the triangle inequality of Wasserstein distance in the theory of optimal transport. We refer to \cite[Lemma 5.5]{Santambrogio-Optimal:15} for a proof.

\begin{Lemma}[gluing of measures]\label{lemma:gluing}
Let $(X_i, \mu_i)$ be probability spaces for $i=1,2,3$, with $X_i\subset\mRd$, and $\gamma_{1,2} \in \Pi(\mu_1, \mu_2)$, $\gamma_{2,3} \in \Pi(\mu_2, \mu_3)$. Then there exists (at least) one measure $\gamma \in P(X_1 \times X_2 \times X_3)$ such that $(\pi_{1,2})_{\#} \gamma =\gamma_{1,2}$ and $(\pi_{2,3})_{\#} \gamma =\gamma_{2,3}$, where $\pi_{i,j}$ is the projection defined as $\pi_{i,j}(x_1,x_2,x_3) = (x_i,x_j)$. 
\end{Lemma}

We use this lemma to glue the measures $\mu,\nu$ with their approximations $\mu_h,\nu_h$ as follows. Given  optimal transport plans $\gamma_h\in\Pi(\mu_h,\nu_h)$ between $\mu_h$ and $\nu_h$, $\alpha_h \in \Pi(\mu, \mu_h)$ between $\mu$ and $\mu_h$, and $\beta_h \in \Pi(\nu_h, \nu)$ between $\nu_h$ and $\nu$, we let $\Gamma_h \in P(X_1 \times X_2 \times X_3 \times X_4)$ be a probability measure so that
\begin{equation}\label{eq:Gamma_h}
  (\pi_{1,2})_{\#} \Gamma_h = \alpha_h,
  \quad
  (\pi_{2,3})_{\#} \Gamma_h = \gamma_h,
  \quad
  (\pi_{3,4})_{\#} \Gamma_h = \beta_h,
\end{equation}  
where $X_1=X_2=X$ and $X_3=X_4=Y$; \Cref{lemma:gluing} guarantees the existence of $\Gamma_h$. This in particular implies that projections $\pi_i$ on the coordinate $x_i$ satisfy
\begin{equation}\label{eq:marginals}
(\pi_1)_\# \Gamma_h = \mu,
\quad
(\pi_2)_\# \Gamma_h = \mu_h,
\quad
(\pi_3)_\# \Gamma_h = \nu_h,
\quad
(\pi_4)_\# \Gamma_h = \nu,
\end{equation}
along with $\sigma:=(\pi_{1,4})_\# \Gamma_h \in \Pi(\mu,\nu)$.
Moreover, the following formula holds
\begin{equation}\label{eq:Gammah-int}
 \int_{X_1 \times X_2 \times X_3 \times X_4} F(x_1,x_2) d\Gamma_h(x_1,x_2,x_3,x_4)
  = \int_{X_1 \times X_2} F(x_1,x_2) d\alpha_h(x_1,x_2), 
\end{equation}
for $(x_1,x_2;\alpha_h)$ and similar expressions are valid for
$(x_2,x_3;\gamma_h)$, $(x_3,x_4;\beta_h)$ and $(x_1,x_4;\sigma)$. If $X=Y=\mRd$, and
  $\mu,\mu_h,\nu_h,\nu\in P_2(\mRd)$, then $\alpha_h,\beta_h,\gamma_h,\sigma\in P_2(\mR^{2d})$ according to \eqref{p-moment}.

Now we state and prove our main perturbation estimates for transport plans.
\begin{Theorem}[perturbation of optimal transport plans]\label{th:OT_stability}
Let $T = \nabla \varphi$ be the optimal transport map from $\mu$ to $\nu$ and 
$(\mu,\nu,\varphi)$ be $\lambda-$regular for $\lambda>0$. Let $\mu_h, \nu_h \in P_2(\mRd)$ be approximations of $\mu, \nu$, let $\alpha_h \in \Pi(\mu, \mu_h), \beta_h \in \Pi(\nu_h, \nu)$ be two transport plans between $\mu,\mu_h$ and $\nu_h,\nu$ with $L^2$-errors
\begin{equation}\label{eq:OT_stability_eh}
e_{\alpha_h} := \l( \int_{X^2} |x - x'|^2 d\alpha_h(x,x') \r)^{\frac{1}{2}},
\quad
e_{\beta_h} := \l( \int_{Y^2} |y' - y|^2 d\beta_h(y',y) \r)^{\frac{1}{2}},
\end{equation}
and let $e_h := e_{\alpha_h} + e_{\beta_h}$. If $\gamma_h$ is an optimal transport plan between $\mu_h$ and $\nu_h$, then
	\begin{equation}\label{eq:OT_stability2}
	\l( \int_{X \times Y} |T(x') - y'|^2 d\gamma_h(x',y') \r)^{\frac{1}{2}} \le 
	2 \lambda^{\frac{1}{2}} \ e_h^{\frac{1}{2}} \Big( W_2(\mu,\nu) + e_h \Big)^{\frac{1}{2}} + \lambda e_{\alpha_h} + e_{\beta_h}.
	\end{equation}
\end{Theorem}

\begin{remark}[interpretation of \eqref{eq:OT_stability2}]\label{rem:interpret-OT_stability}
	We emphasize that, in view of \Cref{rem:l2-error-plan} (stability of transport maps), the left hand side of \eqref{eq:OT_stability2} can be viewed as the square of the $L^2$-error between the optimal map $T$ and the transport plan $\gamma_h$. The quantity $e_h$ is a measure of the approximation errors between $\mu,\mu_h$ and $\nu,\nu_h$. A typical choice of $\alpha_h, \beta_h$ is to take the optimal transport plans between $\mu,\mu_h$ and $\nu_h,\nu$ respectively, which implies $e_h = W_2(\mu,\mu_h) + W_2(\nu_h,\nu)$. 
\end{remark}

\begin{remark}[non-uniqueness]\label{rem:non-unique-discrete}
In general, an optimal transport plan $\gamma_h$ between $\mu_h$ and $\nu_h$ may be neither unique nor induced by a map, no matter how small $e_h$ is; for instance, if $X=\{(\pm 1,0)\}_{i=1}^2, Y=\{(0,\pm 1)\}_{j=1}^2$, then any plan from $X$ to $Y$ has the same cost. However, \eqref{eq:OT_stability2} in \Cref{th:OT_stability} shows that all optimal transport plans are somewhat close to the map $T$ in a certain $L^2$-sense dictated by $e_h^{1/2}$.
\end{remark}

\begin{proof}[Proof of \Cref{th:OT_stability}]
The proof proceeds as in \cite[Corollary 3.4]{Gigli-Holder:11}. Let $\Gamma_h\in P_2(X^2 \times Y^2)$ be a measure gluing $\mu, \mu_h, \nu_h, \nu$ and satisfying \eqref{eq:Gamma_h} and \eqref{eq:marginals}. Since $\sigma = (\pi_{1,4})_{\#} \Gamma_h \in \Pi(\mu,\nu)$ is a transport plan between $\mu$ and $\nu$, \Cref{prop:OT_stability_fix_measures} (stability of transport plans) gives
	\begin{equation}\label{eq:proof_OT_stability_1}
	\int_{X \times Y} |T(x) - y|^2 d\sigma(x,y) \le
	\lambda \l( \int_{X \times Y} |x - y|^2 d\sigma(x,y) - W_2^2(\mu, \nu) \r),
	\end{equation}
	according to the definition of Wasserstein distance $W_2^2(\mu,\nu) = \int_{X} |T(x)-x|^2 d\mu(x)$ and the fact that $T$ is the optimal map from $\mu$ to $\nu$. Since $\gamma_h= (\pi_{2,3})_\# \Gamma_h \in \Pi(\mu_h,\nu_h)$ is an optimal transport plan between $\mu_h$ and $\nu_h$, i.e.
	\[
	W_2^2(\mu_h,\nu_h) = \int_{X \times Y} |x'-y'|^2 d\gamma_h(x',y'),
	\]
        applying the triangle inequality, we deduce
        \begin{equation}\label{E:W-triangle}
        W_2(\mu_h,\nu_h) \le W_2(\mu_h,\mu) + W_2(\mu,\nu) + W_2(\nu,\nu_h)
        \le W_2(\mu,\nu) + e_h.
        \end{equation}
        Combining the relation $\sigma=(\pi_{1,4})_\#\Gamma_h$ between $\sigma$ and $\Gamma_h$, namely 
	\begin{align*}
        \l( \int_{X \times Y} |x - y|^2 d\sigma(x,y) \r)^{\frac12} =    
	\l( \int_{X^2 \times Y^2} |x - y|^2 d\Gamma_h(x,x',y',y) \r)^{\frac{1}{2}},
        \end{align*}
        with the triangle inequality yields
        \begin{align*}
        \l( \int_{X \times Y} |x - y|^2 d\sigma(x,y) \r)^{\frac12}   
	& \le \l( \int_{X^2 \times Y^2} |x - x'|^2 d\Gamma_h(x,x',y',y) \r)^{\frac{1}{2}} \\
        & \quad + \l( \int_{X^2 \times Y^2} |x' - y'|^2 d\Gamma_h(x,x',y',y) \r)^{\frac{1}{2}} \\ & \quad + \l( \int_{X^2 \times Y^2} |y' - y|^2 d\Gamma_h(x,x',y',y) \r)^{\frac{1}{2}}.
        \end{align*}
        In view of \eqref{eq:Gamma_h} and \eqref{E:W-triangle}, this expression can be equivalently rewritten as
        \begin{align*}
        \l( \int_{X \times Y} |x - y|^2 d\sigma(x,y) \r)^{\frac12} 
	& \le \l( \int_{X \times Y} |x - x'|^2 d\alpha_h(x,x') \r)^{\frac{1}{2}} \\
        & \quad + \l( \int_{X \times Y} |x' - y'|^2 d\gamma_h(x',y') \r)^{\frac{1}{2}} \\
	& \quad + \l( \int_{X \times Y} |y' - y|^2 d\beta_h(y',y) \r)^{\frac{1}{2}} 
        \le W_2(\mu, \nu) + 2e_h.
	\end{align*}
	This implies the inequality
	\[
	\int_{X \times Y} |x - y|^2 d\sigma(x,y) - W_2^2(\mu, \nu)
	\le 4 W_2(\mu, \nu)e_h + 4 e_h^2,
	\]
	whence substitution  into \eqref{eq:proof_OT_stability_1} immediately gives
\begin{equation}\label{eq:OT_stability1}
\l( \int_{X \times Y} |T(x) - y|^2 d\sigma(x,y) \r)^{\frac{1}{2}} \le 
	2 \lambda^{\frac12} \ e_h^{\frac{1}{2}} \Big( W_2(\mu,\nu) + e_h \Big)^{\frac{1}{2}}.
\end{equation}
	
To prove \eqref{eq:OT_stability2}, we first recall \eqref{eq:Gamma_h} and \eqref{eq:marginals} to write
\begin{align*}
  \int_{X \times Y} |T(x') - y'|^2 d\gamma_h(x',y') &=
  \int_{X^2 \times Y^2} |T(x') - y'|^2 d\Gamma_h(x,x',y',y),
  \\
  \int_{X \times Y} |T(x) - y|^2 d\sigma(x,y) &=
  \int_{X^2 \times Y^2} |T(x) - y|^2 d\Gamma_h(x,x',y',y),
\end{align*}
and next utilize the triangle inequality together with the Lipschitz property \eqref{E:Lip-T} of the optimal transport map $T$ to obtain
\begin{align*}
	&\l( \int_{X \times Y} |T(x') - y'|^2 d\gamma_h(x',y') \r)^{\frac{1}{2}} - \l( \int_{X \times Y} |T(x) - y|^2 d\sigma(x,y) \r)^{\frac{1}{2}}\\
	&\le \l( \int_{X^2 \times Y^2} |T(x') - T(x)|^2 d\Gamma_h(x,x',y',y) \r)^{\frac{1}{2}} +  \l( \int_{X^2 \times Y^2} |y - y'|^2 d\Gamma_h(x,x',y',y) \r)^{\frac{1}{2}} \\
  & \le \lambda \l( \int_{X \times Y} |x' - x|^2 d\alpha_h(x,x') \r)^{\frac{1}{2}} +  \l( \int_{X \times Y} |y - y'|^2 d\beta_h(y',y) \r)^{\frac{1}{2}} = \lambda e_{\alpha_h} + e_{\beta_h}.
\end{align*}
	Rearranging the above inequality and using \eqref{eq:OT_stability1} proves \eqref{eq:OT_stability2}.
\end{proof}

One could also consider a plan $\gamma_h \in \Pi(\mu_h, \nu_h)$ which is not exactly optimal but close to the optimal transport plan. For this case, we prove the following corollary.
\begin{Corollary}[perturbation of transport plans in $L^2$]\label{cor:OT_stability}
Let $\mu,\nu,\mu_h,\nu_h$ and $\varphi$ be as in \Cref{th:OT_stability} (perturbation of optimal transport plans). Let $\gamma_h\in\Pi(\mu_h,\nu_h)$ be a transport plan between $\mu_h$ and $\nu_h$, and let
$\wt{e}_h$ be given by
\begin{equation}\label{eq:OT_stability_epsilon}
  \wt{e}_h := e_h + \frac{\ve_h}{2},
  \quad
  \ve_h := \l( \int_{X \times Y} |x' - y'|^2 d\gamma_h(x',y') \r)^{\frac{1}{2}} - W_2(\mu_h, \nu_h), 
\end{equation}
and $e_h=e_{\alpha_h} + e_{\beta_h}$ be defined in \eqref{eq:OT_stability_eh}. We then have
	\begin{equation}\label{eq:OT_stability2-2}
	\l( \int_{X \times Y} |T(x') - y'|^2 d\gamma_h(x',y') \r)^{\frac{1}{2}} \le 
	2 \lambda^{\frac12} \ \wt{e}_h^{\frac{1}{2}} \Big( W_2(\mu,\nu) + \wt{e}_h \Big)^{\frac{1}{2}} + \lambda e_{\alpha_h} + e_{\beta_h}.
	\end{equation}
\end{Corollary}

\begin{proof}
We argue as in \Cref{th:OT_stability}, except that instead of \eqref{E:W-triangle} we now have
	\begin{equation*}
	\l( \int_{X \times Y} |x' - y'|^2 d\gamma_h(x',y') \r)^{\frac{1}{2}} \le
	W_2(\mu, \nu) + e_h + \ve_h,
	\end{equation*}
whence \eqref{eq:OT_stability1} becomes 
\begin{equation}\label{eq:OT_stability2-1}
\l( \int_{X \times Y} |T(x) - y|^2 d\sigma(x,y) \r)^{\frac{1}{2}} \le 
	2 \, \lambda^{\frac12} \, \wt e_h^{\frac{1}{2}} \Big( W_2(\mu,\nu) + \wt e_h \Big)^{\frac{1}{2}}.
\end{equation}  
The rest of the proof continues as in \Cref{th:OT_stability}.
\end{proof}

Instead of measuring the $L^2$-error in \eqref{eq:OT_stability2-2} between the optimal transport map $T$ from $\mu$ to $\nu$ and the optimal transport plan $\gamma_h$ from $\mu_h$ to $\nu_h$, we could alternatively characterize the discrepancy between the corresponding plans $\gamma$ and $\gamma_h$ in terms of the Wasserstein distance. This is possible in light of \eqref{p-moment} because $\gamma, \gamma_h \in P_2(\mR^{2d})$ provided $\mu,\nu\in P_2(\mRd)$ and $\mu_h,\nu_h\in P_2(\mRd)$, respectively.

\begin{Corollary}[perturbation of transport plans in Wasserstein metric]\label{cor:Wass_prod_space}
Let $T = \nabla \varphi$ be the optimal transport map from $\mu$ to $\nu$ and $(\mu,\nu,\varphi)$ be $\lambda-$regular for $\lambda > 0$. Let $\gamma = (id,T)_\# \mu$ be the optimal transport plan induced by $T$. Let $\mu_h, \nu_h \in P_2(\mRd)$ be approximations to $\mu,\nu$ and $\gamma_h \in \Pi(\mu_h, \nu_h)$ be a transport plan between $\mu_h$ and $\nu_h$. Then we have
	\[
	W_2(\gamma, \gamma_h) \le 2 \, \lambda^{\frac12} \,
	\wt{e}_h^{\frac{1}{2}} \Big( W_2(\mu,\nu) + \wt{e}_h \Big)^{\frac{1}{2}} + e_h,
	\]
	where $e_h = W_2(\mu, \mu_h) + W_2(\nu, \nu_h)$ and $\wt{e}_h$ is defined in \eqref{eq:OT_stability_epsilon}.
\end{Corollary}

\begin{proof}
Let $\alpha_h \in \Pi(\mu, \mu_h), \beta_h = \Pi(\nu_h, \nu)$ be the optimal transport plans for the corresponding OT, and $\Gamma_h\in P_2(\mR^4)$ be a gluing measure satisfying \eqref{eq:Gamma_h} and \eqref{eq:marginals}. We need to construct a suitable transport plan $\wt{\Gamma}_h \in\Pi (\gamma,\gamma_h)$ from $\gamma$ to $\gamma_h$ and use the fact that the corresponding transport cost dominates $W_2(\gamma, \gamma_h)$. To this end, we introduce the map $S: \mR^{4d} \to \mR^{2d}$ defined as $S(x,x',y',y) := \l(x, T(x)\r)$, and consider the push-forward of measure $\Gamma_h$ through $S$
\[        S_{\#} \Gamma_h = (id,T)_{\#} \big( (\pi_{1})_{\#} \Gamma_h \big) = (id,T)_{\#} \mu = \gamma.
\]
Since $(\pi_{2,3})_{\#} {\Gamma}_h = \gamma_h$ according to \eqref{eq:Gamma_h}, we infer that the map $\widetilde{S}:=(S, \pi_{2,3} \big):\mR^{4d} \to \mR^{4d}$ defined by
\[
\widetilde{S}(x,x',y',y) = (z_1,z_2),
\quad
z_1 = (x,T(x)),
\quad
z_2 = (x',y')
\]
induces a push-forward measure $\wt{\Gamma}_h := \widetilde{S}_{\#} \Gamma_h \in \Pi(\gamma,\gamma_h)$ of $\Gamma_h$ through $\widetilde{S}$ that happens to be a transport plan between $\gamma$ and $\gamma_h$. This implies
	\begin{align*}
	W_2(\gamma, \gamma_h) &\le \l( \int_{\mR^{4d}} |z_1 - z_2|^2 \; d\wt{\Gamma}_h(z_1, z_2) \r)^{\frac{1}{2}} \\
	&= \l( \int_{\mR^{4d}} \big| S(x,x',y',y) - \pi_{2,3}(x,x',y',y)  \big|^2 \; d{\Gamma}_h(x,x',y',y) \r)^{\frac{1}{2}} \\
	&= \l( \int_{\mR^{4d}} \big( |x-x'|^2 + |T(x) - y'|^2 \big) \; d{\Gamma}_h(x,x',y',y) \r)^{\frac{1}{2}}.
	\end{align*}

Using the triangle inequality along with \eqref{eq:Gammah-int} yields
\begin{align*}
W_2(\gamma, \gamma_h) &\le \l( \int_{\mR^{4d}} |x-x'|^2 d{\Gamma}_h(x,x',y',y) \r)^{\frac{1}{2}}
\\ & + \l( \int_{\mR^{4d}} |y' - y|^2 \ d{\Gamma}_h(x,x',y',y) \r)^{\frac{1}{2}}
\\ & + \l( \int_{\mR^{4d}} |T(x) - y|^2 \ d{\Gamma}_h(x,x',y',y) \r)^{\frac{1}{2}}
\\ & = W_2(\mu,\mu_h) + W_2(\nu,\nu_h) + \l( \int_{\mR^{4d}} |T(x) - y|^2 \ d{\Gamma}_h(x,x',y',y) \r)^{\frac{1}{2}},
\end{align*}
because $\alpha_h$ and $\beta_h$ are optimal transport plans.
For the last term we recall that $\sigma=(\pi_{1,4})_\#\Gamma_h$ and employ \eqref{eq:Gammah-int} again, together with \eqref{eq:OT_stability2-1}, to arrive at
\begin{align*}
\l( \int_{\mR^{4d}} |T(x) - y|^2 \ d{\Gamma}_h(x,x',y',y) \r)^{\frac{1}{2}}
& = \l( \int_{\mR^{2d}} |T(x) - y|^2 d\sigma(x,y) \r)^{\frac{1}{2}}
\\ & \le 
	2 \lambda^{\frac12} \ \wt{e}_h^{\frac{1}{2}} \Big( W_2(\mu,\nu) + \wt{e}_h \Big)^{\frac{1}{2}}.
\end{align*}
This finishes the proof of the corollary.
\end{proof}

In the next two sections we explain how to use the preceding stability estimates to obtain some useful error estimates for numerical solutions of OT between the probability measures $\mu,\nu$ with densities $f,g$, supported in $\overline{X}, \overline{Y} \subset \mRd$, respectively, 
and quadratic cost $c(x,y) = |x-y|^2$.

\section{Error Estimates for Semi-Discrete Schemes} \label{S:EE-semi}
In this section, we present and prove error estimates for semi-discrete schemes such as those in \cite{Aurenhammer-Minkowski:98, Kitagawa-Convergence:16, Levy-Numerical:15, Merigot-Multiscale:11}. We conclude our discussion with comparisons between our results and those of Berman \cite{Berman-Convergence:18}.

Let $\mu_h = \sum_{i=1}^N f_i \delta_{x_i}$ ($f_i > 0$) be an approximation of $\mu$ satisfying $W_2(\mu, \mu_h) \le h$. One way to obtain such an approximation is explained in \Cref{S:Numerical-OT}. The semi-discrete schemes are designed to compute the optimal transport plan between the {\it discrete} measure $\mu_h$ and the {\it continuous} measure $\nu$. The semi-discrete problem is equivalent to finding a nodal function $\varphi_h := \big(\varphi_h(x_i)\big)_{i=1}^N \in \mR^N$ so that
\begin{equation}\label{eq:semi-discrete-psi}
\nu\l(F_i \r) = f_i, \quad F_i := \partial \varphi_h(x_i) \cap Y \qquad i = 1, \cdots, N,
\end{equation}
where $Y = \{y \in \mRd: g(y) > 0 \}$ and the discrete subdifferential set is given by
\[
\partial \varphi_h(x_i) := \{y \in \mRd: \quad \varphi_h(x_i) + (x_j - x_i) \cdot y \le \varphi_h(x_j) \quad \forall j = 1, \cdots, N\}.
\]
The convex sets $F_i$ are called {\it Laguerre cells (or generalized Voronoi cells)}.
We do not discuss methods to determine $\varphi_h$ for which we refer to  \cite{Aurenhammer-Minkowski:98, Kitagawa-Convergence:16, Levy-Numerical:15, Merigot-Multiscale:11}.
As pointed out in \cite[Lemma 5.3]{Berman-Convergence:18}, vector $\varphi_h \in \mR^N$ is unique up to a constant if $Y$ is a Lipschitz domain, which is a direct consequence of results in \cite{Kitagawa-Convergence:16}. Once we have found $\varphi_h$, the optimal transport map from $\nu$ to $\mu_h$ is given by the map $S_h$
\[
S_h(y) := x_i \quad  \textrm{a.e. } y \in F_i,
\]
and the induced optimal transport plan $\gamma_h \in P(X \times Y)$ reads $\gamma_h = (S_h, id)_{\#} \nu$ and
\begin{equation}\label{semi-discrete-opt-plan}
\gamma_h(E) = \nu \big\{y\in Y: \big( S_h(y),y)\in E  \big\}
\quad\Rightarrow\quad
\gamma_h\big(\{x_i\} \times F_i  \big) = \nu(F_i)
\end{equation}
for all Borel sets $E \subset X\times Y$.
We also point out that
there is no optimal map from $\mu_h$ to $\nu$ since every node $x_i$ is associated with the Laguerre cell $F_i$, which has infinite number of points (multivalued function). We are now ready to estimate the difference of $\gamma_h$ and the continuous optimal map $T$.

\begin{Theorem}[convergence rate for semi-discrete schemes]\label{th:EE-Semi}
Let the triple $(\mu,\nu,\varphi)$ be $\lambda-$regular for $\lambda > 0$ where the measures $\mu$ to $\nu$ have non-negative densities $f$ and $g$. Let $T = \nabla \varphi$ be the optimal transport map from $\mu$ to $\nu$. Let $\mu_h = \sum_{i=1}^N f_i \delta_{x_i}$ be an approximation of $\mu$ with $f_i > 0$. If $\varphi_h$ is a nodal function that solves \eqref{eq:semi-discrete-psi} and $\gamma_h$ is the corresponding semi-discrete optimal plan \eqref{semi-discrete-opt-plan}, then we have
	\begin{equation}\label{eq:EE-Semi-1}
	\int_{X\times Y} |T(x) - y|^2 d\gamma_h(x,y) = \sum_{i=1}^N \int_{F_i} |T(x_i) - y|^2 g(y)dy  \le E^2_h,
	\end{equation}
	where
	\begin{equation}\label{eq:OT-SemiError-Eh}
	E_h := 2 \lambda^{\frac12} \ W_2(\mu, \mu_h)^{\frac{1}{2}} \Big( W_2(\mu,\nu) + W_2(\mu, \mu_h) \Big)^{\frac{1}{2}} + \lambda W_2(\mu, \mu_h). 
	\end{equation}
Moreover, if $\mu_h$ satisfies $W_2(\mu, \mu_h) \le h$, then there exists a constant $C$ depending on $(\mu, \nu)$ such that for $h \le \lambda^{-1}$ we have
	\begin{equation}\label{eq:EE-Semi-h1}
	\int_{X \times Y} |T(x) - y|^2 d\gamma_h(x,y) = \sum_{i=1}^N \int_{F_i} |T(x_i) - y|^2 g(y)dy  \le C\lambda h.
	\end{equation}
\end{Theorem}

\begin{proof}
  The first equality in \eqref{eq:EE-Semi-1} follows from the relation between $\gamma_h=(S_h,id)_\#\nu$ and $\varphi_h$ expressed in \eqref{eq:semi-discrete-psi} and
    \eqref{semi-discrete-opt-plan}
	\begin{equation*}
	\begin{aligned}
	\int_{X \times Y} |T(x) &- y|^2 d\gamma_h(x,y)
	= \int_{Y} |T\l(S_h(y)\r) - y|^2 d\nu(y) \\
	&= \sum_{i=1}^N \int_{F_i} |T\l(S_h(y)\r) - y|^2 d\nu(y) 
	= \sum_{i=1}^N \int_{F_i} |T(x_i) - y|^2 g(y)dy.
	\end{aligned}
	\end{equation*}
	Now we resort to \Cref{th:OT_stability} (perturbation of optimal transport plans). Since there is no discretization of $\nu$, we let $\beta_h = (id,id)_{\#} \nu$ and $\alpha_h$ be the optimal transport plan between $\mu$ and $\mu_h$. Then \eqref{eq:OT_stability2} in \Cref{th:OT_stability} implies that
	\begin{equation*}
	\int_{X \times Y} |T(x') - y'|^2 d\gamma_h(x',y') \le \l( 2 \, \lambda^{\frac12} \, e_h^{\frac{1}{2}} \Big( W_2(\mu,\nu) + e_h \Big)^{\frac{1}{2}} + \lambda e_{\alpha_h} + e_{\beta_h} \r)^2,
	\end{equation*}
	with $e_h=e_{\alpha_h} = W_2(\mu, \mu_h)$ and $e_{\beta_h} = 0$. This proves \eqref{eq:EE-Semi-1}, whereas \eqref{eq:EE-Semi-h1} follows directly from $W_2(\mu,\mu_h) \le h$  and $h\le\lambda^{-1}$.
\end{proof}

We recall that the optimal transport plan $\alpha_h=(id,U_h)_{\#\mu}$ between $\mu$ and $\mu_h$ defined before \eqref{map-mu-muh} satisfies $W_2(\mu, \mu_h) \le h$ according to \eqref{map-mu-muh}. In addition, 
  our estimates \eqref{eq:EE-Semi-1} and \eqref{eq:EE-Semi-h1} contain geometric information about the Laguerre cells $F_i$ provided $g$ is strictly positive. To see this, for each $1\le i\le N$ we introduce the {\it center of mass (or barycenter)} $m_i$ of $F_i$ with respect to the measure $\nu$
\begin{equation}\label{eq:barycenter}
  m_i := \frac{1}{f_i} \int_{F_i} y \, g(y) \, dy.
\end{equation}

\begin{Corollary}[convergence rate for center of mass and Laguerre cells]\label{C:OT-EE-Semi}
We make the same hypotheses as in \Cref{th:EE-Semi}, and further assume that $g(y) \in [c_1, c_2]$ for all $y \in Y$ for some $c_2 \ge c_1 > 0$ and that $Y$ is a bounded convex set. Then there exists a constant $C$ depending on $d$ and $c_1, c_2$ such that
	\begin{equation}\label{eq:EE-Semi-cor1}
	\sum_{i=1}^N f_i \Big( |T(x_i) - m_i|^2 + C \diam( F_i)^2 \Big) \le E^2_h,
	\end{equation}
where $F_i, m_i$ and $E_h$ are defined in \eqref{eq:semi-discrete-psi}, \eqref{eq:barycenter}, and \eqref{eq:OT-SemiError-Eh}, respectively. Moreover, if $\mu_h$ satisfies $W_2(\mu, \mu_h) \le h$, then there exists a constant $C$ depending on $(d, \mu, \nu, c_1, c_2)$ such that for $h \le \lambda^{-1}$ we have
	\begin{equation}\label{eq:EE-Semi-cor2}
	\sum_{i=1}^N f_i \; |T(x_i) - m_i|^2 \le C\lambda h, \quad 
	\sum_{i=1}^N f_i \; \diam\l( F_i \r)^2 \le C\lambda h.
	\end{equation}
\end{Corollary}

\begin{proof}
  Using the facts that $f_i = \nu(F_i) = \int_{F_i} g(y) dy$ and $\int_{F_i} (m_i - y) g(y)dy = 0$, which are valid due to \eqref{eq:semi-discrete-psi} and \eqref{eq:barycenter}, we infer that
\[  
\int_{F_i} \langle T(x_i) - m_i, m_i - y \rangle  g(y)dy = 0,
\]
whence
\begin{equation*}
	\begin{aligned}
	\int_{F_i} |T(x_i) - y|^2 g(y)dy 
	&= \int_{F_i} \Big( |T(x_i) - m_i|^2  + |m_i - y|^2 \Big) g(y)dy\\
	&= f_i|T(x_i) - m_i|^2 + \int_{F_i} |m_i - y|^2 g(y)dy.
	\end{aligned}
	\end{equation*}
	Therefore to prove \eqref{eq:EE-Semi-cor1} it remains to show that there exists a constant $C$ such that
	$$
	C f_i \; \diam\l( F_i \r)^2  \le \int_{F_i} |m_i - y|^2 g(y)dy.
	$$
	Notice that $F_i = \partial \varphi_h(x_i) \cap Y$ is convex because both $ \partial \varphi_h(x_i)$ and $Y$ are convex. Let
	$\wt{m}_i := |F_i|^{-1} \int_{F_i} y dy$ be the center of mass of $F_i$ with respect to the Lebesgue measure, and apply \cite[Theorem 1.8.2]{Gutierrez-MA:2016} to deduce the existence of an ellipsoid $E$ centered at $\wt{m}_i$ such that $d^{-3/2} E \subset F_i \subset E$, where $\alpha E$ denotes the $\alpha$-dilation of $E$ with respect to $\wt{m}_i$. Then we have
	$$
	\begin{aligned}
	\int_{F_i} |m_i -y|^2 g(y)dy 
	&\ge c_1 \int_{F_i} |m_i -y|^2 dy 
	\ge c_1 \int_{F_i} |\wt{m}_i -y|^2 dy \\
	& \ge c_1 \int_{d^{-3/2} E} |\wt{m}_i -y|^2 dy \ge
	c_d c_1 |E| \; \diam(E)^2,
	\end{aligned}
	$$
	where $c_d$ is a constant depending on $d$. In fact, we will prove below that for any ellipsoid $E_0$ centered at $z$, we have
        \begin{equation}\label{E:ellipsoid}
	\int_{E_0} |y-z|^2 dy \ge c'_d |E_0| \; \diam(E_0)^2,
	\end{equation}
	where $c'_d$ is a constant depending on $d$. We finally utilize the properties
	$$
	f_i = \int_{F_i} g(y)dy \le c_2 |E|,
	$$
	and $\mbox{diam}(E) \ge \mbox{diam}(F_i)$ because $F_i \subset E$, to obtain
	$$
	c_d c_1 |E| \; \diam(E)^2 \ge \frac{c_d c_1}{c_2} f_i \diam(F_i)^2.
	$$
	This finishes our proof of \eqref{eq:EE-Semi-cor1} by choosing $C = c_d c_1 c_2^{-1}$. Finally, \eqref{eq:EE-Semi-cor2} is a direct consequence of \eqref{eq:EE-Semi-cor1} because $E_h^2 \le C\lambda h$ whenever $W_2(\mu,\mu_h) \le h$ and $h\le \lambda^{-1}$.

It remains to show the geometric property \eqref{E:ellipsoid}. Let $E_0$ be given by
\[
E_0 = \Big\{y\in\mathbb{R}^d: \quad \sum_{j=1}^d  \frac{x_i^2}{a_i^2} \le 1 \Big\}
\]
and $a_1 = \max_{1\le j \le d} a_j$; hence $\diam(E_0)=2a_1$ and the barycenter $z$ of $E_0$ coincides with the origin. We consider now the subset of $E_0$
\[
\wt{E}_0 = \Big\{ y\in E_0: \quad \frac{1}{3} a_1 \le y_1 \le \frac{2}{3} a_1  \Big\},
\]
which satisfies
\[
\wt{E}_0 \supset \Big\{y_1 \in\mathbb{R}: \quad \frac{1}{3} a_1 \le y_1 \le \frac{2}{3} a_1  \Big\} \times
\Big\{(y_j)_{j=2}^d: \quad \sum_{j=2}^d \frac{y_j^2}{a_j^2} \le \frac{5}{9}  \Big\} = \wt{E}_1.
\]
We see that $E_0 \subset E_1 = [-a_1,a_1] \times \big\{(y_j)_{j=2}^d: \sum_{j=2}^d \frac{y_j^2}{a_j^2} \le 1 \big\}$, whence $|E_0| < |E_1| \le c_d |\wt{E}_0|$, the last inequality being a consequence of a scaling argument between $E_1$ and $\wt{E}_1$. Since $|y|\ge \frac{a_1}{3} = \frac{1}{6} \diam (E_0)$ in $\wt{E}_0$, we infer that
\[
\int_{E_0} |y|^2 \, dy \ge \int_{\wt{E}_0} |y|^2 \, dy
\ge \frac{1}{36} \diam(E_0)^2 |\wt{E}_0| \ge c'_d \diam(E_0)^2 |E_0|.
\]
This shows \eqref{E:ellipsoid} and concludes the proof.
\end{proof}

\begin{remark}[discrete transport map]
To compensate for not having an optimal transport map from $\mu_h$ to $\nu$, because $(x_i\mapsto F_i)_{i=1}^N$ is a multivalued function, we could define the discrete transport map $T_h$ to be
\[
T_h (x_i) = m_i \quad\forall \, i=1, \cdots, N.
\]
Expression \eqref{eq:EE-Semi-cor2} immediately yields the weighted $L^2$ error estimate
\[
\sum_{i=1}^N f_i \big| T(x_i) - T_h(x_i) \big|^2 \le C \lambda h.
\]
\end{remark}  

\begin{remark}[mean errors]
	We point out that since $\sum_{i=1}^N f_i = 1$, \eqref{eq:EE-Semi-cor2} shows
	that on average the pointwise error $|T(x_i) - m_i|^2$ and $\diam(F_i)^2$ are of order $O(h)$. 	
\end{remark}

\begin{remark}[approximate transport plans]
In practice, the semi-discrete schemes may not solve \eqref{eq:semi-discrete-psi} exactly but rather approximately: let $\wt \varphi_h$ solve
$$
\nu(\wt{F}_i) = \wt{f}_i, \quad \wt{F_i}:=\partial \wt\varphi_h(x_i) \cap Y
\qquad i = 1, \cdots, N,
$$
with $\wt f_i$ close to $f_i$, and induce an optimal plan $\wt{\gamma}_h$ between $\wt{\mu}_h = \sum_{i=1}^N \wt{f}_i \delta_{x_i}$ and $\nu$. We note that the error estimate \eqref{eq:EE-Semi-h1} in \Cref{th:EE-Semi} (convergence rate for semi-discete schemes) only requires $W_2(\mu, \wt{\mu}_h) \le Ch$, and thus deduce that as long as the approximate measure $\wt{\mu}_h$ satisfies
$W_2(\mu_h, \wt{\mu}_h) \le Ch$, we are still able to obtain
$$
\sum_{i=1}^N \int_{\wt{F}_i}|T(x_i) - y|^2 g(y)dy  \le Ch,
$$
because $W_2(\mu, \wt{\mu}_h) \le W_2(\mu, \mu_h)  + W_2(\mu_h, \wt{\mu}_h) \le Ch$. This shows that enforcing \eqref{eq:semi-discrete-psi} exactly may not be computationally profitable.
\end{remark}

We conclude this section comparing our results of \Cref{th:EE-Semi} and \Cref{C:OT-EE-Semi} with that of Berman \cite[Theorem 5.4]{Berman-Convergence:18}. The nodal function $\varphi_h=(\varphi_h(x_i))_{i=1}^N$ that satisfies \eqref{eq:semi-discrete-psi} also induces the convex function $\Gamma(\varphi_h)$, in fact the convex envelope of $(\varphi_h(x_i))_{i=1}^N$. For convenience we still denote $\Gamma(\varphi_h)$ by $\varphi_h$ and observe that, being piecewise linear, it dictates a partition $\mathcal{T}_h$ of $X$ into simplices with vertices $(x_i)_{i=1}^N$. However, this partition $\mathcal{T}_h$ of $X$ might not be shape-regular in general even for a quasi-uniform distribution of nodes $(x_i)_{i=1}^N$; we refer to \cite[Section 2.2]{NochZhang-MA:2016}. Therefore, there is no direct relation between elements $K_j$ of $\mathcal{T}_h$ containing $x_i$ and the Voronoi cells $V_i$ in \eqref{eq:Voronoi}. By assuming that $\lambda^{-1} I \le D^2 \varphi \le \lambda I$ for some $\lambda > 0$ and the density $g$ satisfies $g(y) \ge c_1 > 0$ for all $y\in Y$, Berman obtains the following error estimate for the piecewise constant function $\nabla\varphi_h$ over $\mathcal{T}_h$ \cite[Theorem 5.4]{Berman-Convergence:18}
\begin{equation}\label{eq:Berman-EE}
\|\nabla\varphi - \nabla\varphi_h\|_{L^2(X)} \le C h^{\frac12},
\end{equation}
where the constant $C$ depends on $\lambda, c_1$ and other information. We see that this rate of convergence is similar to those in \Cref{th:EE-Semi} and \Cref{C:OT-EE-Semi} but the error notion is different. To explore this fact, we rewrite \eqref{eq:Berman-EE} as follows
\[
\sum_{K_j \in \mathcal{T}_h} \int_{K_j} |\nabla\varphi(x) - \nabla\varphi_h(x) |^2 dx \le Ch,
\]
Since the Hessian $D^2 \varphi$ is uniformly bounded both from above and below, we infer that $\nabla \varphi$ and $(\nabla \varphi)^{-1}$ are both Lipschitz with Lipschitz constants proportional to $\lambda$ and $\lambda^{-1}$. Therefore the previous inequality is equivalent to
\[
\sum_{K_j \in \mathcal{T}_h} \int_{K_j} |x - (\nabla \varphi)^{-1} \nabla\varphi_h(x) |^2 dx \le C \lambda^2 h.
\]
Let $z_j = |K_j|^{-1} \int_{K_j} x dx$ be the barycenter of $K_j$ with respect to the Lebesgue measure. An argument similar to the proof of \Cref{C:OT-EE-Semi} yields
\[
|K_j| \, |z_j - (\nabla \varphi)^{-1}y_j|^2 + |K_j| \diam(K_j)^2 \lesssim
\int_{K_j} |x - (\nabla \varphi)^{-1} \nabla\varphi_h(x) |^2 dx
\]
where $y_j = \nabla\varphi_h(x)$ is the constant slope for $x \in K_j$. Exploiting again that $T=\nabla\varphi$ is Lipschitz, we see that \eqref{eq:Berman-EE} leads to
\begin{equation}\label{eq:Berman-EE2}
  \sum_{K_j \in \mathcal{T}_h} |K_j| \, \big|T(z_j) - y_j\big|^2 \le C\lambda^4 h,
  \quad
  \sum_{K_j \in \mathcal{T}_h} |K_j| \diam(K_j)^2 \le C \lambda^2 h.
\end{equation}
This looks similar to \eqref{eq:EE-Semi-cor2} but involving simplices $K_j$ instead of Laguerre cells $F_i$; \eqref{eq:EE-Semi-cor2} and \eqref{eq:Berman-EE2} are, however, intrinsically different. Several comments are in order.

\begin{enumerate}[$\bullet$]
\item
The Laguerre cell $F_i$ is the convex hull of all vectors $y_j=\nabla\varphi|_{K_j}$ where $K_j\in\mathcal{T}_h$ are the simplices containing $x_i$, the so-called star (or patch) $\omega_i$ \cite[Section 5.3]{NochettoZhang:2018}. Therefore $m_i$ in \eqref{eq:EE-Semi-cor2} can be replaced by any $y_j$ for $K_j\subset\omega_i$ because of the occurrence of $\diam (F_i)$ in \eqref{eq:EE-Semi-cor2}. However, there is no immediate relation between $f_i=\nu(F_i)$ and $|\omega_i|$ or $\mu(\omega)$.

\item
Both estimates \eqref{eq:EE-Semi-cor2} and \eqref{eq:Berman-EE2} require a lower positive bound for $g$. However, \eqref{eq:EE-Semi-h1} just entails $\lambda$-regularity of $(\mu,\nu,\varphi)$ but not a lower bound on either $g$ or $f$; note that the assumption $f_i>0$ in \Cref{th:EE-Semi} is for convenience because if $f_i=0$ we could simply drop the Dirac mass at $x_i$.

\item
The estimate \eqref{eq:EE-Semi-h1} applies to discrete transport plans and extends to fully discrete schemes; see Section \ref{S:EE-fully}. It is not clear what a fully discrete version of \eqref{eq:Berman-EE2} could be if the measure $\nu$ is further discretized into a sum of Dirac measures.

\item
The derivation of \Cref{th:EE-Semi} and \Cref{C:OT-EE-Semi} is purely analytical and hinges on the quantitative stability estimates of Gigli \cite{Gigli-Holder:11}. In contrast, the proof of \eqref{eq:Berman-EE2} in \cite{Berman-Convergence:18} uses a complexification argument to deduce estimates from well-known inequalities in K\"ahler geometry and pluripotential theory.

\item
Both \eqref{eq:EE-Semi-cor2} and \eqref{eq:Berman-EE2} are meaningful ways to measure the error between the continuous optimal map $T$ and the semi-discrete one $T_h$. Our error notion is natural in dealing with optimal transport maps and plans, while the error notion in \cite{Berman-Convergence:18} is perhaps more natural in the setting of \MA equations.
  
\end{enumerate}

\section{Error Estimates for Fully-Discrete Schemes} \label{S:EE-fully}

If we approximate both measures $\mu, \nu$ by discrete measures $\mu_h = \sum_{i=1}^N f_i \delta_{x_i}$ and $\nu_h = \sum_{j=1}^M g_j \delta_{y_j}$, respectively, then we can consider the following {\it linear programming problem} to approximate the optimal transport map $T$ \cite{Benamou-Iterative:15, Cuturi-Sinkhorn:13, Oberman-Efficient:15, Schmitzer-Sparse:16, Schmitzer-Hierarchical:13}: find the discrete measure $\gamma_h=\sum_{i=1}^N\sum_{j=1}^M\gamma_{h,ij} \delta_{x_i} \delta_{y_j}$ such that
\begin{equation}\label{eq:pb-fully-discrete2}
\min_{\gamma_{h,ij}} \; \sum_{i,j=1}^{N,M} \gamma_{h,ij} \, c_{ij} : \quad
\gamma_{h,ij} \ge 0, \;\; \sum_{i=1}^N \gamma_{h,ij} = g_{j}, \;\; \sum_{j=1}^M \gamma_{h,ij} = f_{i},
\end{equation}
where $c_{ij} = |x_i - y_j|^2$. One thing worth pointing out here is that the solution of the above problem may not be unique even though the continuous problem has a unique optimal map $T$ (and plan $\gamma$); see \Cref{rem:non-unique-discrete} (non-uniqueness). The minimum transport cost of \eqref{eq:pb-fully-discrete2} is equal to $W^2_2(\mu_h, \nu_h)$ according to the definition, but in practice we may not get an exact optimal plan $\gamma_h$, whence the cost  $\sum_{i,j} \gamma_{h,ij} c_{ij}$ might be larger than $W^2_2(\mu_h, \nu_h)$. 

Since the discrete plan $\gamma_h$ may not be induced by a map, i.e. for some $i$ there might exist more than one $j$ such that $\gamma_{h,ij} > 0$, one may need to define a map $T_h$ from $\gamma_h$ to approximate the optimal transport map $T$ between $\mu$ and $\nu$. One way is to use a barycentric projection \cite[Definition 5]{Oberman-Efficient:15} and define $T_h$ to be
\begin{equation}\label{eq:def-fully-T_h}
T_h(x_i) := \frac{1}{f_i} \sum_{j=1}^M \gamma_{h,ij} \, y_j.
\end{equation}
This is a discrete counterpart of \eqref{eq:barycenter}.
In general, the quantity $T_h(x_i)$ may not belong to the set $\{y_j: j=1,\cdots,M\}$. The following theorem quantifies the errors committed in approximating the optimal plan $\gamma$ by $\gamma_h$ and the optimal map $T$ by the map $T_h$ generated from $\gamma_h$.

\begin{Theorem}[convergence rate for fully-discrete schemes]\label{T:OT-EE-Fully}
Let the triple $(\mu,\nu,\varphi)$ be $\lambda-$regular for $\lambda > 0$ where the measures $\mu$ to $\nu$ have non-negative densities $f$ and $g$. Let $T = \nabla \varphi$ be the optimal transport map from $\mu$ to $\nu$. Let the approximations $\mu_h = \sum_{i=1}^N f_i \delta_{x_i}$, $\nu_h = \sum_{j=1}^M g_j \delta_{y_j}$ of $\mu, \nu$ satisfy $W_2(\mu,\mu_h), W_2(\nu,\nu_h) \le C_1h$. If $\gamma_h = \sum_{i,j=1}^{N,M}\gamma_{h,ij} \delta_{x_i} \delta_{y_j}$ $\in \Pi(\mu_h, \nu_h)$ is a discrete transport plan so that
    \[
    \l(\sum_{i,j=1}^{N,M} \gamma_{h,ij} \, c_{ij}\r)^{\frac12} - W_2(\mu_h, \nu_h) \le C_2 h,
    \]
then for $h \le \min\{\lambda,\lambda^{-1}\}$ we have
	\begin{equation}\label{eq:EE-Fully-gammah-1}
	\l( \int_{X \times Y} |T(x) - y|^2 d\gamma_h(x,y) \r)^{\frac{1}{2}} =  \l( \sum_{i,j=1}^{N,M} \gamma_{h,ij} \, |T(x_i) - y_j|^2 \r)^{\frac{1}{2}} \le C\lambda^{\frac12} h^{\frac12},    
	\end{equation}
	where $C$ is a constant depending on $(\mu, \nu, C_1, C_2)$. In addition, the Wasserstein distance between $\gamma$ and $\gamma_h$ satisfies
	\begin{equation}\label{eq:EE-Fully-gammah-2}
	W_2(\gamma, \gamma_h) \le C \lambda^{\frac12} h^{\frac12},
	\end{equation}
and the map $T_h$ generated by $\gamma_h$ according to \eqref{eq:def-fully-T_h} verifies
	\begin{equation}\label{eq:EE-Fully-Th}
	  \l( \sum_{i=1}^N f_i |T(x_i) - T_h(x_i)|^2 \r)^{\frac{1}{2}} \le C \lambda^{\frac12} h^{\frac12}.
	\end{equation}
\end{Theorem}

\begin{proof}
  The equality in \eqref{eq:EE-Fully-gammah-1} follows from the definition of $\gamma_h=\sum_{i,j=1}^{N,M}\gamma_{h,ij} \delta_{x_i} \delta_{y_j}$. By letting $\alpha_h \in \Pi(\mu,\mu_h), \beta_h \in \Pi(\nu_h, \nu)$ be the corresponding optimal transport plans between $\mu,\mu_h$ and $\nu,\nu_h$, the inequality in \eqref{eq:EE-Fully-gammah-1} is a direct consequence of \eqref{eq:OT_stability2-2} in \Cref{cor:OT_stability} (perturbation of transport plans in $L^2$) because
\[
		e_h = W_2(\mu,\mu_h) + W_2(\nu_h,\nu) \le 2C_1 h, \quad
		\ve_h = \l(\sum_{i,j} \gamma_{h,ij} \, c_{ij}\r)^{\frac12} - W_2(\mu_h, \nu_h) \le C_2 h,
\]
and $h\le\min\{\lambda,\lambda^{-1}\}$.
Similary, \eqref{eq:EE-Fully-gammah-2} follows from \Cref{cor:Wass_prod_space} (perturbation of transport plans in Wasserstein metric). The proof of \eqref{eq:EE-Fully-Th} is a discrete version of \Cref{C:OT-EE-Semi} (convergence rate for center of mass and Laguerre cells). We first notice that \eqref{eq:pb-fully-discrete2}, together with the definition \eqref{eq:def-fully-T_h} of $T_h$, yields
	\[
	\sum_{j=1}^M \gamma_{h,ij} \big( T_h(x_i)-y_j \big) = f_i T_h(x_i) - \sum_{j=1}^M \gamma_{h,ij} \, y_j = 0  \quad \forall i = 1,\cdots,N.
	\]
	This implies the orthogonality relation
        \[
        \sum_{j=1}^M \gamma_{h,ij} \langle T(x_i) - T_h(x_i), T_h(x_i) - y_j \rangle = 0,
        \]
        whence
	\[\begin{aligned}
	\sum_{i=1}^N \sum_{j=1}^M \gamma_{h,ij} \big|T(x_i) - y_j\big|^2 
	= \sum_{i=1}^N \sum_{j=1}^M \gamma_{h,ij} \l( \big|T(x_i) - T_h(x_i)\big|^2 + \big|T_h(x_i) - y_j\big|^2 \r).
	\end{aligned}\]
        Finally, we combine this equality and \eqref{eq:EE-Fully-gammah-1} to arrive at
	\begin{align*}
	C \lambda h &\ge \sum_{i=1}^N \sum_{j=1}^M \gamma_{h,ij} \big|T(x_i) - y_j\big|^2 
	\\ &\ge \sum_{i=1}^N \sum_{j=1}^M \gamma_{h,ij} \big|T(x_i) - T_h(x_i)\big|^2 \\
	&= \sum_{i=1}^N f_i \big|T(x_i) - T_h(x_i)\big|^2.
	\end{align*}
	This shows \eqref{eq:EE-Fully-Th} and thus finishes the proof of this theorem.
\end{proof}

\Cref{T:OT-EE-Fully} (convergence rates for fully-discrete schemes) is the first result known to us which quantitatively measures the errors between $\gamma$ and $\gamma_h$ and between $T$ and $T_h$. The convergence rate $O(h^{1/2})$ coincides with that for semi-discrete schemes developed in \Cref{S:EE-semi} and first proved by Berman \cite{Berman-Convergence:18} for a different error notion. Moreover, \Cref{T:OT-EE-Fully} reveals that it is enough to solve the linear programming problem \eqref{eq:pb-fully-discrete2} approximately provided that the cost for $\gamma_h$ is within $O(h)$ from the optimal cost. Therefore, approximate techniques from \cite{Benamou-Iterative:15, Cuturi-Sinkhorn:13, Oberman-Efficient:15} are relevant in practice.

It is also worth pointing out that, according to \Cref{T:OT-EE-Fully}, although $\gamma_h$ is not generally sparse it must be close to the sparse plan $\wt\gamma_h := (S_h)_{\#\mu_h}$, the push-forward of $\mu_h$ by the map $S_h:=(id,T_h): \{x_i\}_{i=1}^N\subset X \to X\times Y$ and given by
\[  
\wt\gamma_h(A) = \mu_h\big(S_h^{-1}(A)\big)
= \sum_{(x_i,T_h(x_i)) \in A} f_i
\]
for all measurable sets $A \subset X\times Y$. We hope this observation might provide insight on acceleration processes for solving problem \eqref{eq:pb-fully-discrete2} that take advantage of sparsity of $\gamma_h$, as shown for instance in \cite{Oberman-Efficient:15}. 

\bibliographystyle{amsplain}
\bibliography{OT_v6}

\providecommand{\bysame}{\leavevmode\hbox to3em{\hrulefill}\thinspace}
\providecommand{\MR}{\relax\ifhmode\unskip\space\fi MR }
\providecommand{\MRhref}[2]{%
  \href{http://www.ams.org/mathscinet-getitem?mr=#1}{#2}
}
\providecommand{\href}[2]{#2}
\begin{thebibliography}{10}

\bibitem{Angenent-Minimizing:03}
S.~Angenent, S.~Haker, and A.~Tannenbaum, \emph{Minimizing flows for the
  {M}onge--{K}antorovich problem}, SIAM J. Math. Anal. \textbf{35} (2003),
  no.~1, 61--97.

\bibitem{Aurenhammer-Minkowski:98}
F.~Aurenhammer, F.~Hoffmann, and B.~Aronov, \emph{{M}inkowski-type theorems and
  least-squares clustering}, Algorithmica \textbf{20} (1998), no.~1, 61--76.

\bibitem{AzePenotConvex}
D.~Az\'{e} and J.-P. Penot, \emph{Uniformly convex and uniformly smooth convex
  functions}, Ann. Fac. Sci. Toulouse Math. (6) \textbf{4} (1995), no.~4,
  705--730.

\bibitem{Beiglbock-Model;13}
M.~Beiglb{\"o}ck, P.~Henry-Labord{\`e}re, and F.~Penkner,
  \emph{Model-independent bounds for option prices--a mass transport approach},
  Finance Stoch. \textbf{17} (2013), no.~3, 477--501.

\bibitem{Benamou-Computational:00}
J.-D. Benamou and Y.~Brenier, \emph{A computational fluid mechanics solution to
  the {M}onge-{K}antorovich mass transfer problem}, Numer. Math. \textbf{84}
  (2000), no.~3, 375--393.

\bibitem{Benamou-Iterative:15}
J.-D. Benamou, G.~Carlier, M.~Cuturi, L.~Nenna, and G.~Peyr{\'e},
  \emph{Iterative {B}regman projections for regularized transportation
  problems}, SIAM J. Sci. Comput. \textbf{37} (2015), no.~2, A1111--A1138.

\bibitem{Benamou-Minimal:17}
J.-D. Benamou and V.~Duval, \emph{Minimal convex extensions and finite
  difference discretization of the quadratic {M}onge-{K}antorovich problem},
  European J. Appl. Math. \textbf{30} (2019), no.~6, 1041--1078.

\bibitem{Benamou-Numerical:14}
J.-D. Benamou, B.~D. Froese, and A.~M. Oberman, \emph{Numerical solution of the
  optimal transportation problem using the {M}onge--{A}mp{\`e}re equation}, J.
  Comput. Phys. \textbf{260} (2014), 107--126.

\bibitem{Berman-Convergence:18}
R.~J. Berman, \emph{Convergence rates for discretized {Monge--Amp{\`e}re}
  equations and quantitative stability of optimal transport}, arXiv preprint
  arXiv:1803.00785 (2018).

\bibitem{Brenier-Decomposition:87}
Y.~Brenier, \emph{D{\'e}composition polaire et r{\'e}arrangement monotone des
  champs de vecteurs}, CR Acad. Sci. Paris S{\'e}r. I Math. \textbf{305}
  (1987), 805--808.

\bibitem{Brenier-Least;89}
\bysame, \emph{The least action principle and the related concept of
  generalized flows for incompressible perfect fluids}, J. Amer. Math. Soc.
  \textbf{2} (1989), no.~2, 225--255.

\bibitem{Burkard-Assignment:12}
R.~Burkard, M.~Dell'Amico, and S.~Martello, \emph{{A}ssignment {P}roblems,
  revised reprint}, vol. 106, Siam, 2012.

\bibitem{caffarelli1992boundary}
L.~A. Caffarelli, \emph{Boundary regularity of maps with convex potentials},
  Comm. Pure Appl. Math. \textbf{45} (1992), no.~9, 1141--1151.

\bibitem{caffarelli1992regularity}
\bysame, \emph{The regularity of mappings with a convex potential}, J. Amer.
  Math. Soc. \textbf{5} (1992), no.~1, 99--104.

\bibitem{caffarelli1996boundary}
\bysame, \emph{Boundary regularity of maps with convex potentials--{II}}, Ann.
  of Math. (2) \textbf{144} (1996), no.~3, 453--496.

\bibitem{Carlier-Variational;05}
G.~Carlier and F.~Santambrogio, \emph{A variational model for urban planning
  with traffic congestion}, ESAIM Control Optim. Calc. Var. \textbf{11} (2005),
  no.~4, 595--613.

\bibitem{Chiappori-Hedonic;10}
P.-A. Chiappori, R.~J. McCann, and L.~P. Nesheim, \emph{Hedonic price
  equilibria, stable matching, and optimal transport: equivalence, topology,
  and uniqueness}, Econom. Theory \textbf{42} (2010), no.~2, 317--354.

\bibitem{Cullen-Fully;03}
M.~Cullen and H.~Maroofi, \emph{The fully compressible semi-geostrophic system
  from meteorology}, Arch. Ration. Mech. Anal. \textbf{167} (2003), no.~4,
  309--336.

\bibitem{Cullen-Applications;99}
M.~J.~P. Cullen and R.~J. Douglas, \emph{Applications of the {M}onge-{A}mp\`ere
  equation and {M}onge transport problem to meteorology and oceanography},
  Monge {A}mp\`ere equation: applications to geometry and optimization
  ({D}eerfield {B}each, {FL}, 1997), Contemp. Math., vol. 226, Amer. Math.
  Soc., Providence, RI, 1999, pp.~33--53.

\bibitem{Cuturi-Sinkhorn:13}
M.~Cuturi, \emph{Sinkhorn distances: {L}ightspeed computation of optimal
  transport}, Advances in neural information processing systems, 2013,
  pp.~2292--2300.

\bibitem{Figalli-MA:2017}
A.~Figalli, \emph{The {M}onge-{A}mp{\`e}re {E}quation and its {A}pplications},
  Zurich Lectures in Advanced Mathematics, European Mathematical Society (EMS),
  Z\"{u}rich, 2017.

\bibitem{Froese-Numerical:12}
B.~D. Froese, \emph{A numerical method for the elliptic {M}onge--{A}mp{\`e}re
  equation with transport boundary conditions}, SIAM J. Sci. Comput.
  \textbf{34} (2012), no.~3, A1432--A1459.

\bibitem{Gigli-Holder:11}
N.~Gigli, \emph{On {H}{\"o}lder continuity-in-time of the optimal transport map
  towards measures along a curve}, Proc. Edinb. Math. Soc. (2) \textbf{54}
  (2011), no.~2, 401--409.

\bibitem{Gutierrez-MA:2016}
C.~E. Guti\'{e}rrez, \emph{The {M}onge-{A}mp{\`e}re equation}, Progress in
  Nonlinear Differential Equations and their Applications, vol.~89,
  Birkh\"{a}user Springer, 2016, Second edition.

\bibitem{Gutierrez-Refractor;09}
C.~E. Guti{\'e}rrez and Q.~Huang, \emph{The refractor problem in reshaping
  light beams}, Arch. Ration. Mech. Anal. \textbf{193} (2009), no.~2, 423.

\bibitem{Hamfeldt-Viscosity:19}
B.~F. Hamfeldt, \emph{Convergence framework for the second boundary value
  problem for the {M}onge-{A}mp\`ere equation}, SIAM J. Numer. Anal.
  \textbf{57} (2019), no.~2, 945--971.

\bibitem{Kantorovich-translocation:1942}
L.~V. Kantorovich, \emph{On the translocation of masses}, C. R. (Doklady) Acad.
  Sci. URSS (N.S.) \textbf{37} (1942), 199--201.

\bibitem{Kitagawa-Convergence:16}
J.~Kitagawa, Q.~M{\'e}rigot, and B.~Thibert, \emph{Convergence of a {N}ewton
  algorithm for semi-discrete optimal transport}, J. Eur. Math. Soc. (JEMS)
  \textbf{21} (2019), no.~9, 2603--2651.

\bibitem{Levy-Numerical:15}
B.~L\'{e}vy, \emph{A numerical algorithm for {$L_2$} semi-discrete optimal
  transport in 3{D}}, ESAIM Math. Model. Numer. Anal. \textbf{49} (2015),
  no.~6, 1693--1715.

\bibitem{Lindsey-Optimal:17}
M.~Lindsey and Y.~A. Rubinstein, \emph{Optimal transport via a
  {M}onge--{A}mp{\`e}re optimization problem}, SIAM J. Math. Anal. \textbf{49}
  (2017), no.~4, 3073--3124.

\bibitem{Merigot-Multiscale:11}
Q.~M{\'e}rigot, \emph{A multiscale approach to optimal transport}, Computer
  Graphics Forum, vol.~30, Wiley Online Library, 2011, pp.~1583--1592.

\bibitem{Monge-memoire:1781}
G.~Monge, \emph{M{\'e}moire sur la th{\'e}orie des d{\'e}blais et des
  remblais}, Histoire de l'Acad{\'e}mie Royale des Sciences de Paris (1781).

\bibitem{NeilanZhang-MAW2p:2018}
M.~Neilan and W.~Zhang, \emph{Rates of convergence in {$W^2_p$}-norm for the
  {M}onge-{A}mp\`ere equation}, SIAM J. Numer. Anal. \textbf{56} (2018), no.~5,
  3099--3120.

\bibitem{NochettoZhang:2018}
R.~H. Nochetto and W.~Zhang, \emph{Discrete {ABP} estimate and convergence
  rates for linear elliptic equations in non-divergence form}, Found. Comput.
  Math. \textbf{18} (2018), no.~3, 537--593.

\bibitem{NochZhang-MA:2016}
\bysame, \emph{Pointwise rates of convergence for the {O}liker-{P}russner
  method for the {M}onge-{A}mp\`{e}re equation}, Numer. Math. \textbf{141}
  (2019), no.~1, 253--288.

\bibitem{Oberman-Efficient:15}
A.~M. Oberman and Y.~Ruan, \emph{An efficient linear programming method for
  optimal transportation}, arXiv preprint arXiv:1509.03668 (2015).

\bibitem{OlikerPrussner-MA:1989}
V.~I. Oliker and L.~D. Prussner, \emph{On the numerical solution of the
  equation {$(\partial^2z/\partial x^2)(\partial^2z/\partial
  y^2)-((\partial^2z/\partial x\partial y))^2=f$} and its discretizations.
  {I}}, Numer. Math. \textbf{54} (1989), no.~3, 271--293.

\bibitem{Papadakis-Optimal:14}
N.~Papadakis, G.~Peyr{\'e}, and E.~Oudet, \emph{Optimal transport with proximal
  splitting}, SIAM J. Imaging Sci. \textbf{7} (2014), no.~1, 212--238.

\bibitem{Prins-Monge;14}
C.~R. Prins, J.~H.~M. ten Thije~Boonkkamp, J.~Van~Roosmalen, W.~L. Jzerman, and
  T.~W. Tukker, \emph{A {M}onge-{A}mp{\`e}re-solver for free-form reflector
  design}, SIAM J. Sci. Comput. \textbf{36} (2014), no.~3, B640--B660.

\bibitem{Rabin-Regularization;10}
J.~Rabin, J.~Delon, and Y.~Gousseau, \emph{Regularization of transportation
  maps for color and contrast transfer}, Image Processing (ICIP), 2010 17th
  IEEE International Conference on, IEEE, 2010, pp.~1933--1936.

\bibitem{Rubner-Earth;00}
Y.~Rubner, C.~Tomasi, and L.~J. Guibas, \emph{The earth mover's distance as a
  metric for image retrieval}, Int. J. Comput. Vis. \textbf{40} (2000), no.~2,
  99--121.

\bibitem{Santambrogio-Optimal:15}
F.~Santambrogio, \emph{Optimal {T}ransport for {A}pplied {M}athematicians},
  Progress in Nonlinear Differential Equations and their Applications, vol.~87,
  Birkh\"{a}user/Springer, Cham, 2015, Calculus of variations, PDEs, and
  modeling.

\bibitem{Schmitzer-Sparse:16}
B.~Schmitzer, \emph{A sparse multiscale algorithm for dense optimal transport},
  J. Math. Imaging Vision \textbf{56} (2016), no.~2, 238--259.

\bibitem{Schmitzer-Hierarchical:13}
B.~Schmitzer and C.~Schn{\"o}rr, \emph{A hierarchical approach to optimal
  transport}, International Conference on Scale Space and Variational Methods
  in Computer Vision, Springer, 2013, pp.~452--464.

\bibitem{Villani-Optimal:08}
C.~Villani, \emph{Optimal {T}ransport: {O}ld and {N}ew}, Grundlehren der
  Mathematischen Wissenschaften [Fundamental Principles of Mathematical
  Sciences], vol. 338, Springer-Verlag, Berlin, 2009.

\bibitem{Xia-Optimal;03}
Q.~Xia, \emph{Optimal paths related to transport problems}, Commun. Contemp.
  Math. \textbf{5} (2003), no.~02, 251--279.

\end{thebibliography}
\end{document}